\documentclass[11pt]{article}

\usepackage{amssymb,amsmath,amscd,amsthm,xy,enumitem,mathrsfs,pstricks,float,graphicx }
\xyoption{all}

\newtheorem{thm}{Theorem}[section]
\newtheorem{prp}[thm]{Proposition}
\newtheorem{lmm}[thm]{Lemma}

\newtheorem{crl}[thm]{Corollary}

\theoremstyle{definition}

\theoremstyle{remark}
\newtheorem{rmk}[thm]{Remark}

\numberwithin{equation}{section}

\def\lra{\longrightarrow}
\def\Lra{\Longrightarrow}

\def\BE#1{\begin{equation}\label{#1}}
\def\EE{\end{equation}}
\def\lr#1{\langle#1\rangle}

\def\blr#1{\big\langle#1\big\rangle}
\def\ti#1{\tilde{#1}}
\def\wt#1{\widetilde{#1}}
\def\ov#1{\overline{#1}}
\def\eref#1{(\ref{#1})}
\def\tn#1{\textnormal{#1}}
\def\sf#1{\textsf{#1}}

\def\sm#1{\begin{small}#1\end{small}}

\def\De{\Delta}

\def\La{\Lambda}
\def\Om{\Omega}
\def\Si{\Sigma}
\def\Th{\Theta}

\def\io{\iota}

\def\la{\lambda}

\def\om{\omega}

\def\ve{\varepsilon}
\def\ups{\upsilon}

\def\C{\mathbb C}

\def\bfH{\mathbf H}

\def\cM{\mathcal M}

\def\fM{\mathfrak M}

\def\cN{\mathcal N}

\def\P{\mathbb P}
\def\R{\mathbb{R}}

\def\cU{\mathcal U}

\def\Z{\mathbb{Z}}

\def\tnd{\textnormal{d}}

\def\ev{\tn{ev}}

\def\top{\textnormal{top}}

\def\0{\mathbf 0}
\def\1{\mathbf 1}

\def\dbar{\bar\partial}

\def\eset{\emptyset}

\begin{document}

\title{A Recursion for Counts of Real Curves in $\C\P^{2n-1}$:\\
Another Proof}
\author{Penka Georgieva and 
Aleksey Zinger\thanks{Partially supported by  NSF grant DMS 0846978}}
\date{\today}
\maketitle

\begin{abstract}
\noindent
In a recent paper, we obtained a WDVV-type relation for real genus~0 Gromov-Witten invariants 
with conjugate pairs of insertions;
it specializes to a complete recursion in the case of odd-dimensional projective spaces.
This note provides another, more complex-geometric, proof of the latter.
The main part of this approach readily extends to real symplectic manifolds
with empty real locus, but not to the general case.
\end{abstract}


\section{Introduction}
\label{intro_sec}

\noindent
The classical problem of enumerating (complex) rational curves in a complex projective space~$\P^n$ 
is solved in \cite{KM,RT} using the WDVV relation of Gromov-Witten theory.
Over the past decade, significant progress has been made in real enumerative geometry 
and real Gromov-Witten theory.
Invariant signed counts of real rational curves with point constraints
in real surfaces and in many real threefolds are defined in~\cite{Wel1}
and~\cite{Wel2}, respectively.
An approach to interpreting these counts in the style of Gromov-Witten theory,
i.e.~as counts of parametrizations of such curves, is presented in \cite{Cho,Sol}. 
Signed counts of real curves with conjugate pairs of arbitrary (not necessarily point)
constraints in arbitrary dimensions are defined in~\cite{Ge2} and extended
to more general settings in~\cite{Teh}.
Two different WDVV-type relations for the real Gromov-Witten invariants 
of real surfaces as defined in \cite{Cho,Sol}, along with the ideas behind them, 
are stated in~\cite{Sol2};
they yield complete recursions for counts of real rational curves in~$\P^2$ 
as defined in~\cite{Wel1}.
Other recursions for counts of real curves in some real surfaces have 
since been established by completely different methods in \cite{GM,ABL,IKS13a,IKS13b}.\\

\noindent
In \cite{GZ4}, we obtain a WDVV-type relation for real genus~0 Gromov-Witten invariants
with conjugate pairs of constraints without restricting to low-dimensional real
symplectic manifolds. 
In the case of~$\P^{2n-1}$, it specializes to the complete recursions of 
Theorem~\ref{main_thm} and Corollary~\ref{main0_crl}.
These recursions are sufficiently simple to characterize the cases when 
the aforementioned real invariants are nonzero
and thus the existence of real rational curves passing through the specified constraints
is guaranteed; see \cite[Corollary~1.3]{GZ4}.
The main proof of the WDVV-type relation in~\cite{GZ4} is based on
establishing a homology relation on the three-dimensional Deligne-Mumford space $\R\ov\cM_{0,3}$
of genus~0 real curves with 3 conjugate pairs of marked points.
We also give an alternative proof in~\cite{GZ4} which is closer to the proof of \cite[Theorem~10.4]{RT},
but makes use of a conjugate marking.\\

\noindent
In this note, we describe a more complex-geometric variation of the second approach
in~\cite{GZ4}. 
In order to focus on the approach itself, we restrict to~$\P^{2n-1}$, 
but it can be applied in some other cases as well; see Remark~\ref{degen_rmk}.
We work with the explicit system of orientations on the moduli spaces 
of real maps to~$\P^{2n-1}$ defined in \cite[Appendix~A.1]{Teh} 
from an algebro-geometric point of view; the orientations used in~\cite{GZ4} 
are described from the point of view of symplectic topology.
The analysis of the sign of the key gluing map of Lemma~\ref{gluing_lmm}
is carried out in Section~\ref{lmmpf_sec} using polynomials.
The primary motivations for this note are to make  the proof of 
Theorem~\ref{main_thm} and Corollary~\ref{main0_crl} more accessible, 
in particular to algebraic geometers  who may have no interest 
in the general case  of the real WDVV relation of  \cite[Theorem~2.2]{GZ4},
and to highlight the difficulties eliminated by the homology relation 
of \cite[Proposition~4.3]{GZ4}.\\

\noindent
Each odd-dimensional projective space~$\P^{2n-1}$ has two standard anti-holomorphic involutions
(automorphisms of order~2):
\begin{alignat}{2}
\label{taudfn_e}
\tau_{2n}\!:\P^{2n-1}&\lra\P^{2n-1}, &\quad
[z_1,\ldots,z_{2n}]&\lra [\bar{z}_2,\bar{z}_1,\ldots,\bar{z}_{2n},\bar{z}_{2n-1}],\\
\label{etadfn_e}
\eta_{2n}\!:\P^{2n-1}&\lra\P^{2n-1}, &\quad
[z_1,\ldots,z_{2n}]&\lra [-\bar{z}_2,\bar{z}_1,\ldots,-\bar{z}_{2n},\bar{z}_{2n-1}].
\end{alignat}
The fixed locus of the first involution is~$\R\P^{2n-1}$, while 
the fixed locus of the second involution is empty.
Let 
$$\tau\!=\!\tau_2,\,\eta\!=\!\eta_2: \P^1\lra\P^1\,.$$
For $\phi\!=\!\tau_{2n},\eta_{2n}$ and $c\!=\!\tau,\eta$, 
a map $u\!:\P^1\!\lra\!\P^{2n-1}$ is \sf{$(\phi,c)$-real} if
$u\!\circ\!c\!=\!\phi\!\circ\!u$.
For $k\!\in\!\Z^{\ge0}$, a \sf{$k$-marked $(\phi,c)$-real map} is a tuple
$$\big(u,(z_1^+,z_1^-),\ldots,(z_k^+,z_k^-)\big),$$
where $z_1^+,z_1^-,\ldots,z_k^+,z_k^-\!\in\!\P^1$ are distinct points
with $z_i^+\!=\!c(z_i^-)$  and $u$ is a $(\phi,c)$-real map.
Such a tuple is \sf{$c$-equivalent} to another $k$-marked $(\phi,c)$-real map
$$\big(u',(z_1'^+,z_1'^-),\ldots,(z_k'^+,z_k'^-)\big)$$
if there exists a biholomorphic map $h\!:\P^1\!\lra\!\P^1$ such that 
$$h\!\circ\!c=c\!\circ\!h, \qquad u'\!=\!u\!\circ\!h, \quad\hbox{and}\quad
z_i^{\pm}\!=\!h(z_i'^{\pm})~~\forall~i\!=\!1,\ldots,k.$$
If in addition $d\!\in\!\Z^+$, denote by $\fM_{0,k}(\P^{2n-1},d)^{\phi,c}$ 
the moduli space of $c$-equivalence classes of
$k$-marked degree~$d$  holomorphic $(\phi,c)$-real maps.\\

\noindent
By \cite[Theorem~6.5]{Ge2}, a natural compactification
$$\ov\fM_{0,k}(\P^{2n-1},d)^{\tau_{2n},\tau}\supset 
\fM_{0,k}(\P^{2n-1},d)^{\tau_{2n},\tau}$$  
is orientable.
If $d\!\not\in\!2\Z$,
$\ov\fM_{0,k}(\P^{2n-1},d)^{\tau_{2n},\tau}$ has no boundary
and thus carries a $\Z$-homology class;
see \cite[Theorem~1.6]{Ge2}.
By \cite[Lemma~1.9]{Teh},
$$\fM_{0,k}(\P^{2n-1},d)^{\eta_{2n},\tau}=\eset \quad\forall~d\in\Z\,,
\qquad
\fM_{0,k}(\P^{2n-1},d)^{\tau_{2n},\eta}=\eset \quad\forall~d\not\in2\Z\,,$$
and a natural compactification
$$\ov\fM_{0,k}(\P^{2n-1},d)^{\phi,\eta} \supset \fM_{0,k}(\P^{2n-1},d)^{\phi,\eta}$$ 
is orientable for $\phi\!=\!\tau_{2n},\eta_{2n}$.
If $d\!\not\in\!2\Z$, $\ov\fM_{0,k}(\P^{2n-1},d)^{\eta_{2n},\eta}$ has no boundary
and thus carries a $\Z$-homology class;
see \cite[Proposition~1.1]{Teh}.
If $d\!\in\!2\Z$, a glued moduli space
\BE{gluedsp_e}\ov\fM_{0,k}(\P^{2n-1},d)^{\phi} \equiv
\ov\fM_{0,k}(\P^{2n-1},d)^{\phi,\tau}\cup\ov\fM_{0,k}(\P^{2n-1},d)^{\phi,\eta}\EE
is orientable and has no boundary; see \cite[Theorem~1.7]{Teh}
and \cite[Remark~1.11]{Teh}.\\

\noindent
The glued compactified moduli spaces come with natural evaluation maps
$$\ev_i\!: \ov\fM_{0,k}(\P^{2n-1},d)^{\phi}\lra\P^{2n-1},
\qquad 
\big[u,(z_1^+,z_1^-),\ldots,(z_k^+,z_k^-)\big]\lra u(z_i^+).$$
Thus, for $c_1,\ldots,c_k\!\in\!\Z^+$, we define
\BE{realNdfn_e} N_d^{\phi}(c_1,\ldots,c_k)=\int_{\ov\fM_{0,k}(\P^{2n-1},d)^{\phi}}
\ev_1^*H^{c_1}\,\ldots\,\ev_k^*H^{c_k}\in\Z\,,\EE
where $H\!\in\!H^2(\P^{2n-1})$ is the hyperplane class.
For dimensional reasons,
\BE{codimcond_e} N_d^{\phi}(c_1,\ldots,c_k)\neq0
\qquad\Lra\qquad c_1+\ldots+c_k=n(d\!+\!1) -2+k\,. \EE
Similarly to \cite[Lemma~10.1]{RT}, the numbers~\eref{realNdfn_e} are enumerative counts
of real curves in~$\P^{2n-1}$, i.e.~of curves preserved by~$\phi$,
but now with some sign.
They satisfy the usual divisor relation \cite[Section~26.3]{MirSym}.
By \cite[Theorem~1.10]{Teh}, the numbers~\eref{realNdfn_e} with 
$\phi\!=\!\tau_{2n},\eta_{2n}$ vanish if either $d$ or any $c_i$ is even;
see also \cite[Remark~1.11]{Teh}, 
\cite[Corollary~2.5]{GZ4}, and \cite[Theorem~2.7]{GZ4}.\\

\noindent
The nonzero numbers~\eref{realNdfn_e} depend on the chosen orientation of 
the moduli space and are thus well-defined only up to sign,
a priori depending on the degree~$d$.
With the choices in~\cite{Teh},
\BE{taueta_e}N_d^{\tau_{2n}}(c_1,\ldots,c_k)=-N_d^{\eta_{2n}}(c_1,\ldots,c_k)\,;\EE
see \cite[Theorem~1.10]{Teh}.
Thus, it is sufficient to compute the numbers 
\BE{modNums_e}\blr{c_1,\ldots,c_k}_d^{\phi}
\equiv (-1)^{n(d-1)/2}N_d^{\phi}(c_1,\ldots,c_k)\EE
with $\phi\!=\!\eta_{2n}$, $d\!\ge\!1$ odd, and $c_i\!\ge\!3$ odd; 
we comment on the sign modification in Remark~\ref{modNums_rmk}.
For any $d,c_1,\ldots,c_k\!\in\!\Z^+$, let
$$\blr{c_1,\ldots,c_k}_d^{\P^{2n-1}}=\int_{\ov\fM_{0,k}(\P^{2n-1},d)}
\ev_1^*H^{c_1}\,\ldots\,\ev_k^*H^{c_k}\in\Z^{\ge0}\,,$$
where $\ov\fM_{0,k}(\P^{2n-1},d)$ is the usual moduli space of stable (complex) 
$k$-marked genus~0 degree~$d$ holomorphic maps to~$\P^{2n-1}$,
denote the (complex) genus~0 Gromov-Witten invariants of~$\P^{2n-1}$;
they are computed in \cite[Theorem~10.4]{RT}.
Finally, if $c_1,\ldots,c_k\!\in\!\Z$ and $I\!\subset\!\{1,\ldots,k\}$,
let $c_I$ denote a tuple with the entries $c_i$ with $i\!\in\!I$,
in some order.

\begin{thm}\label{main_thm}
Let $\phi\!=\!\tau_{2n},\eta_{2n}$ and $d,k,n,c,c_1,\ldots,c_k\!\in\!\Z^+$.
If $k\!\ge\!2$ and $c_1,\ldots,c_k\!\not\in\!2\Z$,
\begin{equation*}\begin{split}
\blr{c_1,c_2\!+\!2c,c_3,\ldots,c_k}_d^{\phi}
-\blr{c_1\!+\!2c,c_2,c_3,\ldots,c_k}_d^{\phi}
=\sum_{\begin{subarray}{c}2d_1+d_2=d\\ d_1,d_2\ge1 \end{subarray}}
\sum_{I\sqcup J=\{3,\ldots,k\}}
\sum_{\begin{subarray}{c}2i+j=2n-1\\ i,j\ge1 \end{subarray}}\!\!\!\!\!\!
2^{|I|}\Bigg(\qquad&\\
\blr{2c,c_1,c_I,2i}_{d_1}^{\P^{2n-1}}\!\blr{c_2,c_J,j}_{d_2}^{\phi}
-\blr{2c,c_2,c_I,2i}_{d_1}^{\P^{2n-1}}\!\blr{c_1,c_J,j}_{d_2}^{\phi}&\Bigg).
\end{split}\end{equation*}
\end{thm}

\begin{crl}\label{main0_crl}
Let $\phi\!=\!\tau_{2n},\eta_{2n}$ and $d,k,n,c_1,\ldots,c_k\!\in\!\Z^+$.
If $d\!\in\!2\Z$ or $c_i\!\in\!2\Z$ for some~$i$,
$$\blr{c_1,c_2,\ldots,c_k}_d^{\phi}=0.$$
If $k\!\ge\!2$ and $c_1,\ldots,c_k\!\not\in\!2\Z$,
\begin{equation*}\begin{split}
\blr{c_1,c_2,c_3,\ldots,c_k}_d^{\phi}
=d\blr{c_1\!+\!c_2\!-\!1,c_3,\ldots,c_k}_d^{\phi}
+\sum_{\begin{subarray}{c}2d_1+d_2=d\\ d_1,d_2\ge1 \end{subarray}}
\sum_{I\sqcup J=\{3,\ldots,k\}}
\sum_{\begin{subarray}{c}2i+j=2n-1\\ i,j\ge1 \end{subarray}}\!\!\!\!\!\!
2^{|I|}\Bigg(\qquad&\\
d_2\blr{c_1\!-\!1,c_2,c_I,2i}_{d_1}^{\P^{2n-1}}\!\blr{c_J,j}_{d_2}^{\phi}
-d_1\blr{c_1\!-\!1,c_I,2i}_{d_1}^{\P^{2n-1}}\!\blr{c_2,c_J,j}_{d_2}^{\phi}&\Bigg).
\end{split}\end{equation*}
\end{crl} 

\noindent
The formula of Theorem~\ref{main_thm}
immediately implies the recursion of Corollary~\ref{main0_crl},
which in turn determines all numbers $N_d^{\phi}(c_1,\ldots,c_k)$, 
with $\phi\!=\!\tau_{2n},\eta_{2n}$, from the single number
$$\blr{2n\!-\!1}_1^{\tau_{2n}}=N_1^{\tau_{2n}}(2n\!-\!1),$$
i.e.~the number of $\tau_{2n}$-real lines through a 
point in~$\P^{2n-1}$.
The absolute value of this number is of course~1.
With the choice of the orientations as in \cite[Section~5.2]{Teh},
\BE{initNum_e} \lr{2n\!-\!1}_1^{\tau_{2n}}=(-1)^{n-1};\EE
see \cite[Corollary~5.4]{Teh}.
Taking $d\!=\!1$ in Corollary~\ref{main0_crl}, we obtain
$$\blr{c_1,\ldots,c_k}_1^{\tau_{2n}}
=\blr{2n\!-\!1}^{\tau_{2n}}_1=(-1)^{n-1}$$
whenever $c_1,\ldots,c_k\!\in\!\Z^+$ are odd and $c_1\!+\!\ldots\!+\!c_k\!=\!2n\!-\!2\!+\!k$.
Some other numbers obtained from Corollary~\ref{main0_crl} are shown in 
\cite[Tables~1,2]{GZ4}.\\

\noindent
Theorem~\ref{main_thm} follows from Corollary~\ref{main0_crl} by interchanging~$c_2$ and~$c_3$,
which has no effect on the left-hand side of the formula in Corollary~\ref{main0_crl},
and setting the two right-hand sides equal.
Starting as in the proof of \cite[Theorem~10.4]{RT},
we establish the $\phi\!=\!\eta_{2n}$ case of the recursion of Corollary~\ref{main0_crl} 
in Section~\ref{alterpf_sec} as follows.
Denote by $\ov\cM_{0,4}$ the Deligne-Mumford moduli space of stable 
(complex) 4-marked rational curves.
Let
\BE{forgmap_e}\begin{split}
f_{0123}\!: \ov\fM_{0,k+1}(\P^{2n-1},d)^{\eta_{2n}}&\lra\ov\cM_{0,4}\,,\\
\big[(z_0^+,z_0^-),\ldots,(z_k^+,z_k^-),u\big]&\lra[\dot{z}_0^+,\dot{z}_1^+,\dot{z}_2^+,\dot{z}_3^+],
\end{split}\EE
where $[\dot{z}_0^+,\dot{z}_1^+,\dot{z}_2^+,\dot{z}_3^+]\!\in\!\ov\cM_{0,4}$ is the stabilization
of the domain of the stable map with the marked points $z_0^+,z_1^+,z_2^+,z_3^+$ only,
be the morphism forgetting the map to~$\P^{2n-1}$ and all marked points 
other than $z_0^+,z_1^+,z_2^+,z_3^+$.
By adding in $c_3\!=\!1$ if necessary, it can be assumed that $k\!\ge\!3$ 
in Corollary~\ref{main0_crl}.
In Section~\ref{alterpf_sec}, we compare two expressions for the integral 
of the pull-back of the orientation class on~$\ov\cM_{0,4}$ by~$f_{0123}$
over the two-dimensional space of maps passing through 
the constraints $H^{c_1-1},H^1,H^{c_2},\ldots,H^{c_k}$; see~\eref{tiNdfn_e}.
As in the proof of \cite[Theorem~10.4]{RT}, we consider the preimages 
of two different representatives of the point class 
(the Poincare dual of the orientation class): nodal two-component curves, 
with one of them having the 0-th and 1st marked points on a common component
and the other having the 0-th and 2nd marked points on a common component;
see Figure~\ref{m04_fig}, where $\cU\!\lra\!\ov\cM_{0,4}$ denotes the universal curve
and~$\pi$ can be viewed as the cross-ratio
$$\pi\big([z_0,z_1,z_2,z_3]\big)=\frac{z_0-z_2}{z_0-z_3}:\frac{z_1-z_2}{z_1-z_3}\,.$$
The domains of the preimages
of these representatives now have three components, though each preimage is still
encoded by just two of the components.
The number of possible types of the preimages in this case is~7,
instead of~1 as in~\cite{RT};
see Figures~\ref{LHS_fig} and~\ref{RHS_fig}.
In contrast to the proof of \cite[Theorem~10.4]{RT}, the sign of the contribution
of each element in the preimage must be carefully considered;
see Proposition~\ref{sign_prp}.
With the exception of one case (the rightmost diagram in 
Figures~\ref{LHS_fig} and~\ref{RHS_fig}),
each element in the preimage is regular with respect to the restriction of~$f_{0123}$
to the space of  maps meeting the constraints,
with~$f_{0123}$ locally of the~form
$$\C\lra\C, \qquad \ups\lra\ups  \quad\hbox{or}\quad \ups\lra\bar{\ups}\,,$$
with respect to a standard gluing parameter $\ups\!\in\!\C$.
In the exceptional case, each element is the zero set of the map 
$\ups\!\lra\!|\ups|^2$ in some coordinates
and so does not contribute to the curve count.
Setting the sums of all contributions from each of the two degenerations
equal, we obtain Corollary~\ref{main0_crl}.
This approach can also be used to prove \cite[Theorem~2.2]{GZ4} 
whenever the fixed locus of the anti-symplectic involution is empty; 
see Remark~\ref{degen_rmk}.

\begin{figure}
\begin{pspicture} (-1.1,-1.5)(10,2.7)
\psset{unit=.4cm}
\psline[linewidth=.03](2,-1)(33,-1)
\rput(30,-2){$\ov\cM_{0,4}\!\approx\!\P^1$}
\rput(30,4){${\cal U}$}\rput(30.5,2){$\pi$}
\psline[linewidth=.03]{->}(30,3)(30,0)
\pscircle*(5,-1){.2}\pscircle*(15,-1){.2}\pscircle*(25,-1){.2}
\rput(5,-2){$[0,1]$}\rput(15,-2){$[1,0]$}\rput(25,-2){$[1,1]$}
\psline[linewidth=.02](3,6)(7,0)\psline[linewidth=.03](3,0)(7,6)
\pscircle*(5,3){.14}\pscircle*(4.5,2.25){.17}\pscircle*(5.5,3.75){.17}
\pscircle*(5.25,2.625){.17}\pscircle*(4.333,4){.17}
\rput(3.85,2.35){$z_0$}\rput(3.9,3.7){$z_1$}
\rput(6.2,3.6){$z_3$}\rput(6,2.7){$z_2$}
\psline[linewidth=.02](9,5)(9,1)
\psarc[linewidth=.02](8,5){1}{0}{90}\psarc[linewidth=.02](10,1){1}{180}{270}
\psline[linewidth=.02](11,5)(11,1)
\psarc[linewidth=.02](10,5){1}{0}{90}\psarc[linewidth=.02](12,1){1}{180}{270}
\psline[linewidth=.02](20,5)(20,1)
\psarc[linewidth=.02](19,5){1}{0}{90}\psarc[linewidth=.02](21,1){1}{180}{270}
\psline[linewidth=.02](22,5)(22,1)
\psarc[linewidth=.02](21,5){1}{0}{90}\psarc[linewidth=.02](23,1){1}{180}{270}
\psline[linewidth=.02](13,1.5)(18,4)
\psline[linewidth=.02](13,4.29)(18,3.17)
\pscircle*(16.9,3.45){.14}\pscircle*(14.5,2.25){.17}\pscircle*(15.5,3.75){.17}
\pscircle*(15.25,2.625){.17}\pscircle*(14.333,4){.17}
\rput(13.8,2.4){$z_0$}\rput(13.85,3.65){$z_1$}
\rput(16.2,3.9){$z_3$}\rput(15.9,2.4){$z_2$}
\psline[linewidth=.02](24.233,5.05)(24.75,-.275)
\psline[linewidth=.02](25.7,4.65)(24.55,-.525)
\pscircle*(24.7,.15){.14}\pscircle*(24.5,2.25){.17}\pscircle*(25.5,3.75){.17}
\pscircle*(25.25,2.625){.17}\pscircle*(24.333,4){.17}
\rput(23.8,2.4){$z_0$}\rput(23.85,3.65){$z_1$}
\rput(26.2,3.9){$z_3$}\rput(26,2.6){$z_2$}
\end{pspicture}
\caption{The universal curve $\cU\!\lra\!\ov\cM_{0,4}$}
\label{m04_fig}
\end{figure}

\begin{rmk}\label{modNums_rmk}
There are several systematic ways of orienting the moduli spaces $\ov\fM_{0,k}(\P^{2n-1},d)^{\phi,c}$, 
one of which is more natural from the point of view of algebraic geometry and
the others from the point of view of symplectic topology.
In \cite[Section~5.2]{Teh}, these moduli spaces are oriented using 
coefficients of polynomials describing holomorphic maps $\P^1\!\lra\!\P^{2n-1}$;
we use these orientations to define the numbers~\eref{realNdfn_e}
with $\phi\!=\!\tau_{2n}$ and the opposite orientations 
to define the numbers~\eref{realNdfn_e} with $\phi\!=\!\eta_{2n}$
(as needed to orient the glued space~\eref{gluedsp_e} if $d\!\in\!2\Z$).
This choice introduces a sign into the statement of Lemma~\ref{gluing_lmm},
as compared to \cite[Lemma~5.1]{GZ4};
the sign shifts in~\eref{modNums_e} offset the sign of Lemma~\ref{gluing_lmm}.
The orientations of moduli spaces used in~\cite{GZ4} are induced
from various pinching constructions of symplectic topology,
which do not appear as natural in the context of counting curves in projective spaces.
The two systems of orientations on the moduli spaces $\ov\fM_{0,k}(\P^{2n-1},d)^{\phi,c}$ 
agree (up to a sign independent of~$d$) if and only if $n$ is even. 
As explained in Remark~\ref{sign_rmk}, the sign shifts in~\eref{modNums_e}
indirectly switch the two systems of orientations so that \cite[Theorem~2.2]{GZ4}
applies to the numbers~\eref{modNums_e}.
This difference between the two systems of orientations is related to 
a subtle sign issue missed in the description of 
the localization data for real maps to~$\P^{4n+1}$ in the first three versions
of~\cite{Teh}; see Remark~\ref{sign_rmk} for more details.
\end{rmk}

\noindent
In Section~\ref{signcomp_sec}, we compare different orientations of moduli spaces
of constrained real maps and establish Lemma~\ref{orientcomp_lmm3}.
It leads to Corollary~\ref{sign_crl}, which implies Proposition~\ref{sign_prp},
the key step in the proof of Corollary~\ref{main0_crl} in Section~\ref{alterpf_sec}.\\

\noindent
We would like to thank E.~Ionel, J.~Koll\'ar, M.~Liu, N.~Sheridan, J.~Solomon, M.~Tehrani, 
and G.~Tian for related discussions.

\section{Proof of Corollary~\ref{main0_crl}}
\label{alterpf_sec}

\noindent
By~\eref{taueta_e} and the vanishing of the real invariants for $d\!\in\!2\Z$,
it is sufficient to assume that $d$ in Theorem~\ref{main_thm} is odd
and $\phi\!=\!\eta_{2n}$.
Let
$$\ov\fM_k^{\C}(d)=\ov\fM_{0,k}\big(\P^{2n-1},d\big), \qquad
\ov\fM_k^{\R}(d)=\ov\fM_{0,k}\big(\P^{2n-1},d\big)^{\eta_{2n}}\,;$$
we use the same conventions for the uncompactified moduli spaces.
We assume that $k\!\ge\!3$ and
$c_1,\ldots,c_k\!\in\!\Z^+$ are odd and satisfy
the equation on the right-hand side of~\eref{codimcond_e}.
Let
$$f_{0123}\!:\ov\fM_{k+1}^{\R}(d)\lra \ov\cM_{0,4}$$
denote the forgetful morphism in~\eref{forgmap_e}, 
with the marked points on the left-hand side indexed by $0,1,\ldots,k$.
Let
$\Om_{0,4}\!\in\!H^2(\ov\cM_{0,4})$ be the Poincare dual of the point class
and define
\BE{tiNdfn_e}\wt{N}_d^{\phi}(c_1,\ldots,c_k)=
(-1)^{\frac{n(d-1)}{2}}\!\!\!
\int_{\ov\fM_{k+1}^{\R}(d)}\!\!f_{0123}^*\Om_{0,4}\,
\ev_0^*H^{c_1-1}\,\ev_1^*H\,\ev_2^*H^{c_2}\,\ldots\,\ev_k^*H^{c_k}\,.\EE
Choose a generic collection  of linear subspaces $H_0,\ldots,H_k\!\subset\!\P^{2n-1}$ 
of complex codimensions $\hbox{$c_1\!-\!1$},1,c_2,\ldots,c_k$, respectively.
For any $\la\!\in\!\cM_{0,4}$, let 
$$Z_{\la}=\big\{u\!\in\!f_{0123}^{-1}(\la)\!:\,\ev_i(u)\!\in\!H_i~\forall\,
i\!=\!0,1,\ldots,k\big\} \subset \ov\fM_{k+1}^{\R}(d)\,.$$
This set is a compact oriented 0-dimensional submanifold of 
$\ov\fM_{k+1}^{\R}(d)$, i.e.~a finite set of signed points, 
if $\la$ is generic.
The number~\eref{tiNdfn_e} is the signed cardinality $^{\pm}\!|Z_{\la}|$ of this set.\\

\noindent
We prove Corollary~\ref{main0_crl} by explicitly describing the elements of 
$Z_{[1,1]}$ and~$Z_{[1,0]}$, with notation as in Figure~\ref{m04_fig},
and determining their contribution to the number~\eref{tiNdfn_e}.
The domain~$\Si_u$ of each element~$u$ of $Z_{[1,1]}$ and~$Z_{[1,0]}$ consists of 
at least two irreducible components.
Since the fixed point locus of the involution~$\eta_{2n}$ on~$\P^{2n-1}$
is empty, 
$$\ov\fM_{k+1}^{\R}(d) =\ov\fM_{0,k+1}(\P^{2n-1},d)^{\eta_{2n},\eta}$$
and $\Si_u$ has an odd number of irreducible components;
the involution~$\eta_u$ associated with~$u$ restricts to~$\eta$ on one of 
the components and interchanges the others in pairs.
For dimensional reasons, the number of irreducible components of~$\Si_u$ cannot be greater
than~3 and thus must be precisely~3.
Each map~$u$ with its marked points is completely determined by its restriction~$u^{\R}$
to the component~$\Si_u^{\R}$ of~$\Si_u$ preserved by~$\eta_u$ and 
its restriction~$u^{\C}$ to either of the other components.\\

\noindent
We depict all possibilities for the elements of $Z_{[1,1]}$ and~$Z_{[1,0]}$
in Figures~\ref{LHS_fig} and~\ref{RHS_fig}, respectively.
In each of the diagrams, the vertical line represents the irreducible component~$\Si_u^{\R}$ 
of~$\Si_u$
preserved by~$\eta_u$, while the two horizontal lines represent the components of~$\Si_u$
interchanged by~$\eta_u$; the integers next to the lines specify the degrees of~$u$ on
the corresponding components.
The larger dots on the three lines indicate the locations of 
the marked points~$z_0^+,z_1^+,z_2^+,z_3^+$; 
we label them by the codimensions of the constraints they map~to,
i.e.~$c_1\!-\!1,1,c_2,c_3$, in order to make the connection 
with the expression in Corollary~\ref{main0_crl} more apparent.
If a marked point~$z_i^+$ lies on the bottom component, 
its conjugate~$z_i^-$ lies on the top component.
In such a case, we indicate the conjugate point by a small dot on the upper component
and label it with~$\bar{c}_i$;
the restriction of~$u$ to the upper component maps this point to
the linear subspace
$$\ov{H_i}\equiv \eta_{2n}(H_i)\subset\P^{2n-1}.$$
By the definition of~$Z_{[1,1]}$, each diagram in Figure~\ref{LHS_fig}
contains a node separating 
the marked points~$z_0^+,z_1^+$ (i.e.~the larger dots labeled by $c_1\!-\!1,1$) from 
the marked points~$z_2^+,z_3^+$ (i.e.~the larger dots labeled by $c_2,c_3$).
Similarly, each diagram in Figure~\ref{RHS_fig} contains a node separating 
the marked points~$z_0^+,z_2^+$ from the marked points~$z_1^+,z_3^+$.
We arrange the diagrams in both cases so that the pair of marked points containing~0
lies above the other pair.
The remaining marked points, $z_4^{\pm},\ldots,z_k^{\pm}$, are distributed between the three components
in some way.\\

\begin{figure}
\begin{pspicture}(-.5,-3.5)(10,2.8)
\psset{unit=.4cm}
\psline[linewidth=.05](3,5)(3,0)
\psline[linewidth=.02](2.5,4.5)(7.5,4.5)\psline[linewidth=.02](2.5,.5)(7.5,.5)
\pscircle*(4,4.5){.15}\pscircle*(6,4.5){.15}
\pscircle*(3,3.5){.15}\pscircle*(3,1.5){.15}
\rput(4,5){\sm{1}}\rput(6,5){\sm{$c_1\!-\!1$}}
\pscircle*(6,.5){.1}\rput(6,1.1){\sm{$\ov{c_1\!-\!1}$}}
\pscircle*(4,.5){.1}\rput(4,1){\sm{$\bar1$}}
\rput(2.4,3.5){\sm{$c_2$}}\rput(2.4,1.5){\sm{$c_3$}}
\rput(8.2,4.5){\sm{$d_1$}}\rput(3,5.7){\sm{$d_2$}}
\psline[linewidth=.05](12,5)(12,0)
\psline[linewidth=.02](11.5,4.5)(16.5,4.5)\psline[linewidth=.02](11.5,.5)(16.5,.5)
\pscircle*(13,4.5){.15}\pscircle*(15,4.5){.15}
\pscircle*(12,2.5){.15}\pscircle*(14,.5){.15}
\rput(13,5){\sm{1}}\rput(15,5){\sm{$c_1\!-\!1$}}
\pscircle*(15,.5){.1}\rput(15,1.1){\sm{$\ov{c_1\!-\!1}$}}
\pscircle*(13,.5){.1}\rput(13,1){\sm{$\bar1$}}
\rput(11.4,2.5){\sm{$c_2$}}\rput(14.2,-.2){\sm{$c_3$}}
\rput(17.2,4.5){\sm{$d_1$}}\rput(12,5.7){\sm{$d_2$}}
\pscircle*(14,4.5){.1}\rput(14.2,3.8){\sm{$\bar{c}_3$}}
\psline[linewidth=.05](21,5)(21,0)
\psline[linewidth=.02](20.5,4.5)(25.5,4.5)\psline[linewidth=.02](20.5,.5)(25.5,.5)
\pscircle*(22,4.5){.15}\pscircle*(24,4.5){.15}
\pscircle*(21,2.5){.15}\pscircle*(23,.5){.15}
\rput(22,5){\sm{1}}\rput(24,5){\sm{$c_1\!-\!1$}}
\pscircle*(24,.5){.1}\rput(24,1.1){\sm{$\ov{c_1\!-\!1}$}}
\pscircle*(22,.5){.1}\rput(22,1){\sm{$\bar1$}}
\rput(20.4,2.5){\sm{$c_3$}}\rput(23.2,-.2){\sm{$c_2$}}
\rput(26.2,4.5){\sm{$d_1$}}\rput(21,5.7){\sm{$d_2$}}
\pscircle*(23,4.5){.1}\rput(23.2,3.8){\sm{$\bar{c}_2$}}
\psline[linewidth=.05](30,1)(30,-4)
\psline[linewidth=.02](29.5,.5)(34.5,.5)\psline[linewidth=.02](29.5,-3.5)(34.5,-3.5)
\pscircle*(31,.5){.15}\pscircle*(33,.5){.15}
\pscircle*(34,-3.5){.15}\pscircle*(32,-3.5){.15}
\rput(31,1){\sm{1}}\rput(33,1){\sm{$c_1\!-\!1$}}
\pscircle*(33,-3.5){.1}\rput(33,-2.9){\sm{$\ov{c_1\!-\!1}$}}
\pscircle*(31,-3.5){.1}\rput(31,-3){\sm{$\bar1$}}
\rput(34.2,-4.2){\sm{$c_3$}}\rput(32.2,-4.2){\sm{$c_2$}}
\rput(35.2,.5){\sm{$d_1$}}\rput(30,1.7){\sm{$d_2$}}
\pscircle*(32,.5){.1}\rput(32.2,-.2){\sm{$\bar{c}_2$}}
\pscircle*(34,.5){.1}\rput(34.2,-.2){\sm{$\bar{c}_3$}}
\psline[linewidth=.05](3,-3)(3,-8)
\psline[linewidth=.02](2.5,-3.5)(7.5,-3.5)\psline[linewidth=.02](2.5,-7.5)(7.5,-7.5)
\pscircle*(3,-5.5){.15}\pscircle*(6,-3.5){.15}
\pscircle*(7,-7.5){.15}\pscircle*(5,-7.5){.15}
\rput(2.5,-5.5){\sm{1}}\rput(6,-3){\sm{$c_1\!-\!1$}}
\pscircle*(6,-7.5){.1}\rput(6,-6.9){\sm{$\ov{c_1\!-\!1}$}}
\rput(7.2,-8.2){\sm{$c_3$}}\rput(5.2,-8.2){\sm{$c_2$}}
\rput(8.2,-3.5){\sm{$d_1$}}\rput(3,-2.3){\sm{$d_2$}}
\pscircle*(5,-3.5){.1}\rput(5.2,-4.2){\sm{$\bar{c}_2$}}
\pscircle*(7,-3.5){.1}\rput(7.2,-4.2){\sm{$\bar{c}_3$}}
\psline[linewidth=.05](12,-3)(12,-8)
\psline[linewidth=.02](11.5,-3.5)(16.5,-3.5)\psline[linewidth=.02](11.5,-7.5)(16.5,-7.5)
\pscircle*(12,-5.5){.15}\pscircle*(13,-3.5){.15}
\pscircle*(16,-7.5){.15}\pscircle*(14,-7.5){.15}
\rput(13,-3){\sm{1}}\rput(10.7,-5.5){\sm{$c_1\!-\!1$}}
\pscircle*(13,-7.5){.1}\rput(13,-7){\sm{$\bar1$}}
\rput(16.2,-8.2){\sm{$c_3$}}\rput(14.2,-8.2){\sm{$c_2$}}
\rput(17.2,-3.5){\sm{$d_1$}}\rput(12,-2.3){\sm{$d_2$}}
\pscircle*(14,-3.5){.1}\rput(14.2,-4.2){\sm{$\bar{c}_2$}}
\pscircle*(16,-3.5){.1}\rput(16.2,-4.2){\sm{$\bar{c}_3$}}
\psline[linewidth=.05](21,-3)(21,-8)
\psline[linewidth=.02](20.5,-3.5)(25.5,-3.5)\psline[linewidth=.02](20.5,-7.5)(25.5,-7.5)
\pscircle*(21,-6.5){.15}\pscircle*(21,-4.5){.15}
\pscircle*(25,-7.5){.15}\pscircle*(23,-7.5){.15}
\rput(20.5,-4.5){\sm{1}}\rput(19.7,-6.5){\sm{$c_1\!-\!1$}}
\rput(25.2,-8.2){\sm{$c_3$}}\rput(23.2,-8.2){\sm{$c_2$}}
\rput(26.2,-3.5){\sm{$d_1$}}\rput(21,-2.3){\sm{$d_2$}}
\pscircle*(23,-3.5){.1}\rput(23.2,-4.2){\sm{$\bar{c}_2$}}
\pscircle*(25,-3.5){.1}\rput(25.2,-4.2){\sm{$\bar{c}_3$}}
\end{pspicture}
\caption{Domains of elements of $Z_{[1,1]}$}
\label{LHS_fig}
\end{figure}

\begin{figure}
\begin{pspicture}(-.5,-3.5)(10,2.8)
\psset{unit=.4cm}
\psline[linewidth=.05](3,5)(3,0)
\psline[linewidth=.02](2.5,4.5)(7.5,4.5)\psline[linewidth=.02](2.5,.5)(7.5,.5)
\pscircle*(4,4.5){.15}\pscircle*(6,4.5){.15}
\pscircle*(3,3.5){.15}\pscircle*(3,1.5){.15}
\rput(4,5){\sm{$c_2$}}\rput(6,5){\sm{$c_1\!-\!1$}}
\pscircle*(6,.5){.1}\rput(6,1.1){\sm{$\ov{c_1\!-\!1}$}}
\pscircle*(4,.5){.1}\rput(4,1){\sm{$\bar{c}_2$}}
\rput(2.5,3.5){\sm{1}}\rput(2.4,1.5){\sm{$c_3$}}
\rput(8.2,4.5){\sm{$d_1$}}\rput(3,5.7){\sm{$d_2$}}
\psline[linewidth=.05](12,5)(12,0)
\psline[linewidth=.02](11.5,4.5)(16.5,4.5)\psline[linewidth=.02](11.5,.5)(16.5,.5)
\pscircle*(13,4.5){.15}\pscircle*(15,4.5){.15}
\pscircle*(12,2.5){.15}\pscircle*(14,.5){.15}
\rput(13,5){\sm{$c_2$}}\rput(15,5){\sm{$c_1\!-\!1$}}
\pscircle*(15,.5){.1}\rput(15,1.1){\sm{$\ov{c_1\!-\!1}$}}
\pscircle*(13,.5){.1}\rput(13,1){\sm{$\bar{c}_2$}}
\rput(11.5,2.5){\sm{1}}\rput(14.2,-.2){\sm{$c_3$}}
\rput(17.2,4.5){\sm{$d_1$}}\rput(12,5.7){\sm{$d_2$}}
\pscircle*(14,4.5){.1}\rput(14.2,3.8){\sm{$\bar{c}_3$}}
\psline[linewidth=.05](21,5)(21,0)
\psline[linewidth=.02](20.5,4.5)(25.5,4.5)\psline[linewidth=.02](20.5,.5)(25.5,.5)
\pscircle*(22,4.5){.15}\pscircle*(24,4.5){.15}
\pscircle*(21,2.5){.15}\pscircle*(23,.5){.15}
\rput(22,5){\sm{$c_2$}}\rput(24,5){\sm{$c_1\!-\!1$}}
\pscircle*(24,.5){.1}\rput(24,1.1){\sm{$\ov{c_1\!-\!1}$}}
\pscircle*(22,.5){.1}\rput(22,1){\sm{$\bar{c}_2$}}
\rput(20.4,2.5){\sm{$c_3$}}\rput(23,-.2){\sm{1}}
\rput(26.2,4.5){\sm{$d_1$}}\rput(21,5.7){\sm{$d_2$}}
\pscircle*(23,4.5){.1}\rput(23,3.8){\sm{$\bar{1}$}}
\psline[linewidth=.05](30,1)(30,-4)
\psline[linewidth=.02](29.5,.5)(34.5,.5)\psline[linewidth=.02](29.5,-3.5)(34.5,-3.5)
\pscircle*(31,.5){.15}\pscircle*(33,.5){.15}
\pscircle*(34,-3.5){.15}\pscircle*(32,-3.5){.15}
\rput(31,1){\sm{$c_2$}}\rput(33,1){\sm{$c_1\!-\!1$}}
\pscircle*(33,-3.5){.1}\rput(33,-2.9){\sm{$\ov{c_1\!-\!1}$}}
\pscircle*(31,-3.5){.1}\rput(31,-3){\sm{$\bar{c}_2$}}
\rput(34.2,-4.2){\sm{$c_3$}}\rput(32,-4.2){\sm{1}}
\rput(35.2,.5){\sm{$d_1$}}\rput(30,1.7){\sm{$d_2$}}
\pscircle*(32,.5){.1}\rput(32,-.2){\sm{$\bar{1}$}}
\pscircle*(34,.5){.1}\rput(34.2,-.2){\sm{$\bar{c}_3$}}
\psline[linewidth=.05](3,-3)(3,-8)
\psline[linewidth=.02](2.5,-3.5)(7.5,-3.5)\psline[linewidth=.02](2.5,-7.5)(7.5,-7.5)
\pscircle*(3,-5.5){.15}\pscircle*(6,-3.5){.15}
\pscircle*(7,-7.5){.15}\pscircle*(5,-7.5){.15}
\rput(2.4,-5.5){\sm{$c_2$}}\rput(6,-3){\sm{$c_1\!-\!1$}}
\pscircle*(6,-7.5){.1}\rput(6,-6.9){\sm{$\ov{c_1\!-\!1}$}}
\rput(7.2,-8.2){\sm{$c_3$}}\rput(5,-8.2){\sm{1}}
\rput(8.2,-3.5){\sm{$d_1$}}\rput(3,-2.3){\sm{$d_2$}}
\pscircle*(5,-3.5){.1}\rput(5,-4.2){\sm{$\bar{1}$}}
\pscircle*(7,-3.5){.1}\rput(7.2,-4.2){\sm{$\bar{c}_3$}}
\psline[linewidth=.05](12,-3)(12,-8)
\psline[linewidth=.02](11.5,-3.5)(16.5,-3.5)\psline[linewidth=.02](11.5,-7.5)(16.5,-7.5)
\pscircle*(12,-5.5){.15}\pscircle*(13,-3.5){.15}
\pscircle*(16,-7.5){.15}\pscircle*(14,-7.5){.15}
\rput(13,-3){\sm{$c_2$}}\rput(10.7,-5.5){\sm{$c_1\!-\!1$}}
\pscircle*(13,-7.5){.1}\rput(13,-7){\sm{$\bar{c}_2$}}
\rput(16.2,-8.2){\sm{$c_3$}}\rput(14,-8.2){\sm{1}}
\rput(17.2,-3.5){\sm{$d_1$}}\rput(12,-2.3){\sm{$d_2$}}
\pscircle*(14,-3.5){.1}\rput(14,-4.2){\sm{$\bar{1}$}}
\pscircle*(16,-3.5){.1}\rput(16.2,-4.2){\sm{$\bar{c}_3$}}
\psline[linewidth=.05](21,-3)(21,-8)
\psline[linewidth=.02](20.5,-3.5)(25.5,-3.5)\psline[linewidth=.02](20.5,-7.5)(25.5,-7.5)
\pscircle*(21,-6.5){.15}\pscircle*(21,-4.5){.15}
\pscircle*(25,-7.5){.15}\pscircle*(23,-7.5){.15}
\rput(20.4,-4.5){\sm{$c_2$}}\rput(19.7,-6.5){\sm{$c_1\!-\!1$}}
\rput(25.2,-8.2){\sm{$c_3$}}\rput(23,-8.2){\sm{1}}
\rput(26.2,-3.5){\sm{$d_1$}}\rput(21,-2.3){\sm{$d_2$}}
\pscircle*(23,-3.5){.1}\rput(23,-4.2){\sm{$\bar{1}$}}
\pscircle*(25,-3.5){.1}\rput(25.2,-4.2){\sm{$\bar{c}_3$}}
\end{pspicture}
\caption{Domains of elements of $Z_{[1,0]}$}
\label{RHS_fig}
\end{figure}
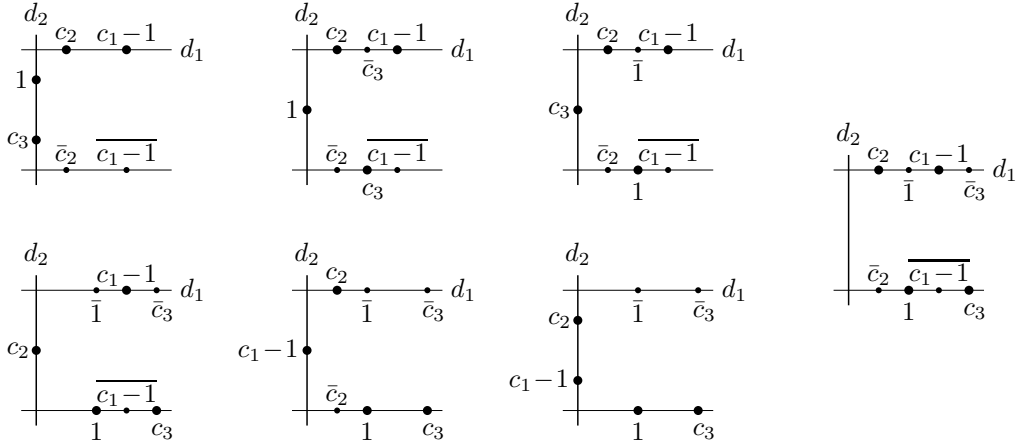

\noindent
Each element~$u$ of $Z_{[1,1]}$ and $Z_{[1,0]}$ corresponds,
via the restriction to~$\Si_u^{\R}$ and the upper component,
to a pair $(u^{\C},u^{\R})$, with
$$[u^{\C}]\in \fM_{k_1+1}(d_1), \quad
[u^{\R}]\in \fM_{k_2+1}^{\R}(d_2),\quad
2d_1\!+\!d_2=d, \quad k_1\!+\!k_2=k\!+\!1,$$
such that $u^{\R}$ and $u^{\C}$ meet at the pair of extra marked points
and pass through $H_0,\ldots,H_k$ or their conjugates as required by 
the distribution of the marked points.
Each such pair $u\!=\!(u^{\C},u^{\R})$ is an isolated element~of
$$\fM_{0,k_1+1}(\P^{2n-1},d_1)\times\fM_{k_2+1}^{\R}(n,d_2)$$
and has a well-defined contribution~$\ve(u)$ to the number~\eref{tiNdfn_e},
i.e.~the signed number of nearby elements of~$Z_{\la}$, with $\la\!\in\!\cM_{0,4}$.
By the next proposition, 
$$\ve(u)=(-1)^{n(d_2-1)/2}$$ 
for all elements~$u$ 
represented by the three diagrams in the first rows 
of Figures~\ref{LHS_fig} and~\ref{RHS_fig},
$$\ve(u)=-(-1)^{n(d_2-1)/2}$$ 
for the three diagrams in the second rows in these figures,
and $\ve(u)\!=\!0$ for the remaining, right-most diagram
in each of the figures.
Even if there were a contribution from the right-most diagram,
it would have been the same for $Z_{[1,1]}$ and $Z_{[1,0]}$
and so would have had no effect on the recursion of Theorem~\ref{main_thm}. 
 
\begin{prp}\label{sign_prp}
Suppose $u\!\in\!Z_{[1,1]}$. 
\begin{enumerate}[label=(\arabic*),leftmargin=*]
\item If $\Si_u^{\R}$ contains either of the marked points~$z_2^+,z_3^+$,
then $\ve(u)\!=\!(-1)^{n(d_2-1)/2}$.
\item If $\Si_u^{\R}$ contains either of the marked points~$z_0^+,z_1^+$,
then $\ve(u)\!=\!-(-1)^{n(d_2-1)/2}$. 
\item If $\Si_u^{\R}$ contains neither of the marked points~$z_0^+,z_1^+,z_2^+,z_3^+$,
then $\ve(u)\!=\!0$. 
\end{enumerate}
The same statements with~1 and~2 interchanged hold for $u\!\in\!Z_{[1,0]}$. 
\end{prp}

\begin{proof} We apply Corollary~\ref{sign_crl} with $k$ replaced by $k\!+\!1$, 
$\{1,2,3,4\}$ by $\{0,1,2,3\}$, and with linear subspaces
of codimensions $c_1\!-\!1,1,c_2,\ldots,c_k$.
We take $J\!\subset\!\{0,1,\ldots,k\}$ to be the subset indexing the pairs of
marked points of~$u$ that lie on the central component~$\Si^{\R}$, 
$I^+$ to be the subset indexing the pairs with the first marked point
on the upper component, i.e.~the domain of~$u^{\C}$,
and $I^-$ to be the complement of $I^+\!\sqcup\!J$ 
in $\{1,\ldots,k\}$.
Since $c_i\!\not\in\!2\Z$ and $0\!\not\in\!I^-$, 
the set on the left-hand side of~\eref{signcases_e} is empty.
Since $\cN_{d_1,d_2;I^+,J,I^-}(\bfH)$ is 0-dimensional in this case,
Corollary~\ref{sign_crl} compares the sign of the elements of $\cN_{d_1,d_2;I^+,J,I^-}(\bfH)$
with the sign of the nearby elements of~$Z_{\la}$.\\

\noindent
Since the first case above corresponds to the first case on 
the right-hand side of~\eref{signcases_e}, the two signs differ 
by $(-1)^{nd_1}$.
Taking into account the extra sign in~\eref{tiNdfn_e}, 
we obtain the first claim of the proposition.
Since the second case above corresponds to the second case on  
the right-hand side of~\eref{signcases_e},
we similarly obtain the second claim.
The final claim of the proposition follows from the last statement 
of Corollary~\ref{sign_crl}.
\end{proof}

\begin{proof}[{\bf\emph{Proof of Corollary~\ref{main0_crl}}}]
We determine the number of elements represented by each diagram 
in Figures~\ref{LHS_fig} and~\ref{RHS_fig}.
Splitting the set $\{4,\ldots,k\}$ into subsets~$I$ and~$J$ in all possible ways,
we put the pairs of marked points indexed by~$J$ on the central component~$\Si_u^{\R}$,
one point of each pair indexed by~$I$ on the top component, and thus the other point
in the pair on the bottom component.
This gives $2^{|I|}$ choices of the distribution and requires~$u^{\C}$ to
pass through either $H_i$, with $i\!\in\!I$, or the conjugate complex hyperplane~$\ov{H}_i$.
By Proposition~\ref{sign_prp}, the contribution~$\ve(u)$ is independent of this choice.
Thus, we can simply multiply the number for one of these distributions by~$2^{|I|}$.
With the constraints completely distributed, we replace the node condition
by the usual splitting of the diagonal, i.e.~an extra constraint of~$H^i$ for~$u^{\C}$
and of~$H^j$ for~$u^{\R}$ with all possible~$i$ and~$j$ so that 
$i\!+\!j\!=\!2n\!-\!1$.
Thus, the contribution to the number~\eref{tiNdfn_e} from each diagram in 
Figures~\ref{LHS_fig} and~\ref{RHS_fig},
each partition $\{4,\ldots,k\}\!=\!I\!\sqcup\!J$, and each partition
$2n\!-\!1\!=\!i\!+\!j$ is
\BE{diagcontr_e}(-1)^{n(d_2-1)/2}\ve\blr{c_{\hat{I}},i}_{d_1}^{\P^{2n-1}}
N_{d_2}^{\eta_{2n}}\big(c_{\hat{J}},j\big)
=\ve\blr{c_{\hat{I}},i}_{d_1}^{\P^{2n-1}}\blr{c_{\hat{J}},j}_{d_2}^{\eta_{2n}},\EE
where 
\begin{enumerate}[label=$\bullet$,leftmargin=*]
\item $\ve\!=\!1$ for the three diagrams in the first rows of the figures,
$\ve\!=\!-1$ for the second rows, and $\ve\!=\!0$ for the right-most diagrams;
\item $\hat{I}$ is the union of $I$ and the subset of $\{0,1,2,3\}$ indexing
the pairs of marked points on the top and bottom components (e.g.~$\{0,1,3\}$
for the second diagram in the first row of Figure~\ref{LHS_fig});
\item $\hat{J}$ is the union of $J$ and the subset of $\{0,1,2,3\}$ indexing
the pairs of marked points on the vertical component (e.g.~$\{2\}$
for the second diagram in the first row of Figure~\ref{LHS_fig});
\item $c_0\!=\!c_1\!-\!1$, with $c_1$ as in~\eref{tiNdfn_e}, and
$c_1\!=\!1$ for the purposes of the definitions of~$c_{\hat{I}}$ 
and~$c_{\hat{J}}$ in~\eref{diagcontr_e}.
\end{enumerate}
Since $c_1\!-\!1\!\in\!2\Z$, the last factor in~\eref{diagcontr_e} vanishes
in the case of the last two diagrams in the second rows of both figures;
see \cite[Theorem~2.7]{GZ4}, 
which is an immediate consequence of Lemma~\ref{orientcomp_lmm3} in this case.
Furthermore, if $d_1\!=\!0$ and the complex invariant in~\eref{diagcontr_e} is nonzero, 
then $|\hat{I}|\!=\!2$ for dimensional reasons;
thus, the only $d_1\!=\!0$ contributions arise from the first diagrams in
Figures~\ref{LHS_fig} and~\ref{RHS_fig} with $I\!=\!\eset$.\\

\noindent
By the previous paragraph, only the three diagrams in the first row
and the first diagram in the second row of each figure contribute to the number~\eref{tiNdfn_e}.
The contribution from Figure~\ref{LHS_fig}, which corresponds to $\la\!=\![1,1]$
in~$\ov\cM_{0,4}$, thus equals
\begin{equation*}\begin{split}
&\blr{c_1,\ldots,c_k}_d^{\eta_{2n}}
+\sum_{\begin{subarray}{c}2d_1+d_2=d\\ d_1,d_2\ge1 \end{subarray}}
\sum_{I\sqcup J=\{4,\ldots,k\}}\sum_{\begin{subarray}{c}2i+j=2n-1\\ i,j\ge1\end{subarray}}
\!\!\!2^{|I|}\Bigg( 
d_1\blr{c_1\!-\!1,c_I,2i}_{d_1}^{\P^{2n-1}}\!\blr{c_2,c_3,c_J,j}_{d_2}^{\eta_{2n}}\\
&\hspace{.8in}
+d_1\blr{c_1\!-\!1,c_3,c_I,2i}_{d_1}^{\P^{2n-1}}\!\blr{c_2,c_J,j}_{d_2}^{\eta_{2n}}
+d_1\blr{c_1\!-\!1,c_2,c_I,2i}_{d_1}^{\P^{2n-1}}\!\blr{c_3,c_J,j}_{d_2}^{\eta_{2n}}\\
&\hspace{3.5in}-d_2\blr{c_1\!-\!1,c_2,c_3,c_I,2i}_{d_1}^{\P^{2n-1}}
\!\blr{c_J,j}_{d_2}^{\eta_{2n}}\Bigg).
\end{split}\end{equation*}
The contribution from Figure~\ref{RHS_fig}, which corresponds to $\la\!=\![1,0]$
in~$\ov\cM_{0,4}$, similarly equals
\begin{equation*}\begin{split}
&d\blr{c_1\!+\!c_2\!-\!1,c_3,\ldots,c_k}_d^{\eta_{2n}}
+\sum_{\begin{subarray}{c}2d_1+d_2=d\\ d_1,d_2\ge1 \end{subarray}}
\sum_{I\sqcup J=\{4,\ldots,k\}}
\sum_{\begin{subarray}{c}2i+j=2n-1\\ i,j\ge1\end{subarray}}\!\!\!2^{|I|}\Bigg(\\
&\hspace{.9in} 
d_2\blr{c_1\!-\!1,c_2,c_I,2i}_{d_1}^{\P^{2n-1}}\!\blr{c_3,c_J,j}_{d_2}^{\eta_{2n}}
+d_2\blr{c_1\!-\!1,c_2,c_3,c_I,2i}_{d_1}^{\P^{2n-1}}\!\blr{c_J,j}_{d_2}^{\eta_{2n}}\\
&\hspace{.8in}
+d_1\blr{c_1\!-\!1,c_2,c_I,2i}_{d_1}^{\P^{2n-1}}\!\blr{c_3,c_J,j}_{d_2}^{\eta_{2n}}
-d_1\blr{c_1\!-\!1,c_3,c_I,2i}_{d_1}^{\P^{2n-1}}\!\blr{c_2,c_J,j}_{d_2}^{\eta_{2n}}
\Bigg).
\end{split}\end{equation*}
Setting the two expressions equal and solving for $\blr{c_1,\ldots,c_k}_d^{\eta_{2n}}$,
we obtain
\begin{equation*}\begin{split}
&\blr{c_1,\ldots,c_k}_d^{\eta_{2n}}= 
d\blr{c_1\!+\!c_2\!-\!1,c_3,\ldots,c_k}_d^{\eta_{2n}}
+\sum_{\begin{subarray}{c}2d_1+d_2=d\\ d_1,d_2\ge1 \end{subarray}}
\sum_{I\sqcup J=\{4,\ldots,k\}}\sum_{\begin{subarray}{c}2i+j=2n-1\\ i,j\ge1\end{subarray}}
\!\!\!2^{|I|}\Bigg(\\
&\hspace{.5in}
2d_2\blr{c_1\!-\!1,c_2,c_3,c_I,2i}_{d_1}^{\P^{2n-1}}
\!\blr{c_J,j}_{d_2}^{\eta_{2n}}
\!-d_1\blr{c_1\!-\!1,c_I,2i}_{d_1}^{\P^{2n-1}}
\!\blr{c_2,c_3,c_J,j}_{d_2}^{\eta_{2n}}\\
&\qquad+
d_2\blr{c_1\!-\!1,c_2,c_I,2i}_{d_1}^{\P^{2n-1}}\!\blr{c_3,c_J,j}_{d_2}^{\eta_{2n}}
-2d_1\blr{c_1\!-\!1,c_3,c_I,2i}_{d_1}^{\P^{2n-1}}\!\blr{c_2,c_J,j}_{d_2}^{\eta_{2n}}\Bigg);
\end{split}\end{equation*}
this formula simplifies to the statement of Corollary~\ref{main0_crl}.
\end{proof}

\begin{rmk}\label{degen_rmk}
The above extends directly to real symplectic manifolds $(X,\om,\phi)$
such that the fixed locus~$X^{\phi}$ of~$\phi$ is empty.
If $X^{\phi}\!\neq\!\eset$, the spaces $Z_{[1,1]}$ and~$Z_{[1,0]}$ 
defined in this section could also contain  two-component maps~$(u_1,u_2)$ of two types:
\begin{enumerate}[label=(\arabic*),leftmargin=*]
\item  the involution on the domain interchanges the two components of the domain
and fixes the node
(this corresponds to sphere bubbling in open GW-theory);
\item the involution on the domain restricts to~$\tau$ on each component of the domain
and fixes the node
(this corresponds to disk bubbling in open GW-theory).
\end{enumerate}
As the above degenerations are of real codimension one,
their intersections with~$Z_{[1,1]}$ and~$Z_{[1,0]}$ are one-dimensional.
Bubble maps of the first type appear in the second proof of \cite[Theorem~2.2]{GZ4}
and do not contribute to the number in \cite[(6.2)]{GZ4} by \cite[Proposition~6.2]{GZ4}.
By the same reasoning, these bubble maps would not contribute to 
the analogue of~\eref{tiNdfn_e} for general real symplectic manifolds as 
in \cite[Theorem~2.2]{GZ4}.
However, the proof of \cite[Proposition~6.2]{GZ4} does not apply 
to the spaces of two-component bubble maps of the second type above,
because they can further degenerate into three real bubbles and 
the function $u\!\lra\!z_1(u)$ in the proof of \cite[Proposition~6.2]{GZ4}
vanishes along some of these degenerations;
the problem degenerations correspond to the intersections of the closures
of different strata of two-component  maps.
The conclusion of \cite[Proposition~6.2]{GZ4} can fail for the closures of
the individual strata of two-component  maps, though 
it may perhaps hold for the connected components of
the union of such closures under the assumptions of \cite[Theorem~2.2]{GZ4}.
For real symplectic manifolds of dimension~6 (and thus with fixed locus of dimension~3),
this issue may be related to some linking phenomena to which an allusion
is made in \cite[Remark~4]{Sol2};
these phenomena do not effect the WDVV relation of \cite[Theorem~2.2]{GZ4} though.
\end{rmk}

\section{Sign computations}
\label{signcomp_sec}

\noindent
For $d\!\in\!\Z^+$, denote~by $\cN_d\subset\ov\fM_0^{\R}(d)$
the sub-orbifold of maps from domains consisting of precisely three components 
and let
\begin{equation*}\begin{split}
\wt\cN_d&=\bigsqcup_{\begin{subarray}{c}2d_1+d_2=d\\ d_1\ge0,d_2>0\end{subarray}}
\!\!\!\!\!\!\wt\cN_{d_1,d_2}\,,\qquad\hbox{where}\\
\wt\cN_{d_1,d_2}&=\big\{(u^{\C},u^{\R})\in
\fM_1^{\C}(d_1)\!\times\!\fM_1^{\R}(d_2)\!:
\ev_0(u^{\C})\!=\!\ev_0(u^{\R})\big\},
\end{split}\end{equation*} 
with the marked points indexed by~0.
Identifying the marked point~$z_0^{\C}$ of the domain of~$u^{\C}$ 
with the first marked point~$z_0^{\R}$ of the domain of~$u^{\R}$
and the marked point~$\eta(z_0^{\C})$ of the map $\eta_{2n}\!\circ\!u^{\C}\!\circ\!\eta$
with~$\eta(z_0^{\R})$, we obtain a double covering
$$q\!:\wt\cN_d \lra \cN_d\,.$$
The canonical orientations of  $\P^{2n-1}$ and $\fM_1^{\C}(d_1)$
and the chosen orientation of $\fM_1^{\R}(d_2)$
induce an orientation on $\wt\cN_d$.
The actions of the deck transformation on the moduli spaces
$\fM_1^{\R}(d_2)$ and $\fM_1^{\C}(d_1)$ and the condition $\ev_0(u^{\C})\!=\!\ev_0(u^{\R})$
are all orientation-reversing because
\begin{enumerate}[label=$\bullet$,leftmargin=*]
\item the first action is the conjugation of the marked point;
\item the second action corresponds to the complex conjugation on 
$\fM_0^{\C}(d_1)$, which is of even complex dimension,
and to the conjugation of the marked point;
\item the third action corresponds to the complex conjugation on~$\P^{2n-1}$.
\end{enumerate}
In particular, $\cN_d$ is not orientable ($\wt\cN_d$ is connected
by \cite[Appendix~A.1]{Teh}).\\

\noindent
Let $L^{\C}\!\lra\!\fM_1^{\C}(d_1)$ and $L^{\R}\!\lra\!\fM_1^{\R}(d)$
be the tautological line bundles and
$$\ti{L}=\pi_1^*L^{\C}\otimes_{\C} \pi_2^*L^{\R}\lra \wt\cN_d\,,$$
where $\pi_1,\pi_2$ are the component projection maps.
The action of the natural lift of the deck transformation on~$\ti{L}$ 
is $\C$-antilinear on each fiber and induces a vector bundle~$L$ over~$\cN_d$. 
This is the normal line bundle of $\cN_d$ in $\ov\fM_0^{\R}(d)$:
as described in Section~\ref{lmmpf_sec}, there is a gluing~map
\BE{Phidfn_e3}\Phi\!:U\lra \ov\fM_0^{\R}(d),\EE
where $U\!\subset\!L$ is a neighborhood of the zero set in~$L$.
The orientation on the total space of~$\ti{L}$ descends to an orientation
on the quotient vector bundle of~$L$. 

\begin{lmm}\label{gluing_lmm}
The restriction of the gluing map~\eref{Phidfn_e3} to 
$\cN_{d_1,d_2}\!\equiv\!q(\wt\cN_{d_1,d_2})$ is orientation-preserving with respect to 
the orientation on the total space of~$L$ described above
if and only if $nd_1\!\in\!2\Z$.
\end{lmm}

\noindent
This lemma is proved in Section~\ref{lmmpf_sec}.
In Corollary~\ref{sign_crl},  
Lemma~\ref{gluing_lmm} is applied to marked moduli spaces over 
topological components $\cN_{d_1,d_2;I^+,J,I^-}$ of~$\cN_d$ on 
which the two conjugate bubbles can be systematically distinguished. 
These topological components are thus oriented by the choice of 
which conjugate bubble is distinguished.
In the case of the diagrams in Figures~\ref{LHS_fig} and~\ref{RHS_fig},
we take the bubble corresponding to the upper line segment to be the distinguished one.

\begin{rmk}\label{sign_rmk}
Moduli spaces in GW-theory are normally oriented by orienting
the index of the linearized $\dbar$-operator via a pinching off construction;
see the proof of \cite[Lemma~5.2]{GZ4}.
In complex GW-theory, the standard orientations
of the index bundles are essentially complex and gluing maps like~\eref{Phidfn_e3} 
are automatically orientation-preserving.
In similarity with the situation in complex GW-theory,
analogous gluing maps are assumed to be orientation-preserving in \cite[(11)]{PSW} without any comment;
this was also assumed to be the case in the early versions of~\cite{Teh}.
In \cite[Section~5.2]{Teh}, the moduli spaces $\ov\fM_0^{\R}(d)$ are oriented directly;
this is the orientation used in this paper as described in Section~\ref{lmmpf_sec}.
If $n\!\in\!2\Z$, the resulting orientation agrees with the orientation
induced by a real square root of $\La_{\C}^{\top}(T\P^{2n-1},\tnd\eta_{2n})$
via the pinching off construction of \cite[Section~2.1]{Teh};
see the paragraph above Remark~6.9 in~\cite{Teh}.
Thus, in this case, the gluing map~\eref{Phidfn_e3} is orientation-preserving,
as Lemma~\ref{gluing_lmm} states.
If \hbox{$n\!\not\in\!2\Z$}, $\La_{\C}^{\top}(T\P^{2n-1},\tnd\eta_{2n})$ does
not admit a real square root.
In this case,  Lemma~\ref{gluing_lmm} implies that 
the chosen orientations of $\ov\fM_0^{\R}(d)$ differ from 
orientations arising via a systematic pinching off construction,
as in \cite[Sections~4,6]{GZ1}, by $(-1)^{n(d-1)/2}$ for $d$ odd.
The implication of this subtle issue for purely computational purposes is 
that the Euler class of the normal bundle of a fixed locus in 
$\ov\fM_k^{\R}(d)$  described in \cite[Section~6.2]{Teh}
should be multiplied by $(-1)^{n(d-d_0)/2}$, where $d_0$ is the degree on the central
component. 
\end{rmk}

\noindent
For $k\!\in\!\Z^+$, there is a fibration
$$\pi\!:\fM_k^{\R}(d)\lra \fM_0^{\R}(d)$$
obtained by forgetting the $k$ pairs of conjugate marked points.
The fiber of~$\pi$ over any point~$[u]$ of the base is isomorphic~to 
an open subspace~of $(\P^1)^k$ by the~map
\begin{gather*}
\io_{k,u}\!:
\big\{(z_1,\ldots,z_k)\!\in\!(\P^1)^k\!:\,z_i\!\neq\!z_j,\eta(z_j)~\forall\,i\!\neq\!j\big\}
\lra \fM_k^{\R}(d), \\
(z_1,\ldots,z_k) \lra \big[(z_1,\eta(z_1)),\ldots,(z_k,\eta(z_k)),u\big].
\end{gather*}
For each subset $I\!\subset\!\{1,\ldots,k\}$, 
let $\io_{k,u;I}$ denote the modification of~$\io_{k,u}$ taking 
the $i$-th component~$z_i$ of $(z_1,\ldots,z_k)$ to the second element
in the $i$-th conjugate pair whenever $i\!\in\!I$; thus, $\io_{k,u;\eset}\!=\!\io_{k,u}$.
The canonical orientations of $\fM_0^{\R}(d)$
and of~$\P^1$ induce via~$\io_{k,u;I}$ an orientation on $\fM_k^{\R}(d)$
and thus on $\ov\fM_k^{\R}(d)$,
which we will call the $I$-orientation.
The orientation on this space, 
which is used to define the numbers~\eref{realNdfn_e}, is the $\eset$-orientation.
Since $\eta$ is an orientation-reversing involution on~$\P^1$,
the $I$-orientation agrees with the canonical orientation if and only if $|I|$
is~even.\\

\noindent
For each $i\!=\!1,\ldots,k$, let
$$\ov\ev_i\equiv\eta_{2n}\!\circ\!\ev_i\!: \ov\fM_k^{\R}(d)\lra \P^{2n-1}$$
be the evaluation map at the second point in the $i$-th conjugate pair.
Denote~by 
\BE{totevdfn_e} \ev\!\equiv\!\ev_1\!\times\!\ldots\!\times\!\ev_k\!:
\ov\fM_k^{\R}(d)\lra (\P^{2n-1})^k\EE
the total evaluation map at the first point in each conjugation pair.
For each $I\!\subset\!\{1,\ldots,k\}$, let
$$\ev_I\!:
\ov\fM_k^{\R}(d)\lra (\P^{2n-1})^k$$
be the modification of~$\ev$ obtained by replacing~$\ev_i$ with~$\ov\ev_i$
whenever $i\!\in\!I$.\\

\noindent
For any subspace $H\!\subset\!\P^{2n-1}$, let $\ov{H}\!=\!\eta_{2n}(H)$ as before.
If $\bfH\!=\!(H_1,\ldots,H_k)$ is a tuple of subspaces of~$\P^{2n-1}$, let
$$\lr\bfH=H_1\times\ldots\times H_k\subset (\P^{2n-1})^k\,.$$
For each $I\!\subset\!\{1,\ldots,k\}$, denote by
$\lr\bfH_I\!\subset\!(\P^{2n-1})^k$ the modification of~$\lr\bfH$
obtained by replacing the $i$-th component~$H_i$ with~$\ov{H_i}$ whenever $i\!\in\!I$.
We define an~involution on $(\P^{2n-1})^k$~by
\begin{gather*}
\Th_I\!: (\P^{2n-1})^k\lra (\P^{2n-1})^k, \qquad 
\big(x_1,\ldots,x_k\big)\lra \big(\Th_{I;1}(x_1),\ldots,\Th_{I;k}(x_k)\big),\\
\hbox{where}\qquad
\Th_{I;i}(x)=\begin{cases}x,&\hbox{if}~i\!\not\in\!I;\\
\eta_{2n}(x),&\hbox{if}~i\!\in\!I.
\end{cases}
\end{gather*}
Let
$$\ov\fM_d^{\R}(\bfH)_I=
\big\{u\!\in\!\ov\fM_k^{\R}(d)_I\!:~\ev_I(u)\!\in\! \lr\bfH_I\big\}.$$
This subspace does not depend on the choice of~$I$, but
its orientation imposed below does in general.\\ 

\noindent
Suppose $\bfH\!=\!(H_1,\ldots,H_k)$ is a tuple of complex linear subspaces of~$\P^{2n-1}$
that are in general position, i.e.~so that the restriction of 
the total evaluation map~\eref{totevdfn_e} to every stratum of the moduli space
(consisting of maps from domains of a fixed topological type) is transverse
to $\lr\bfH$ in~$(\P^{n-1})^k$.
If $I\!\subset\!\{1,\ldots,k\}$, $\ov\fM_d^{\R}(\bfH)_I$ is then a smooth manifold.
The complex orientations on~$\lr\bfH_I$ and~$(\P^{2n-1})^k$,
the $I$-orientation on $\ov\fM_k^{\R}(d)$,
and the map~$\ev_I$ induce an orientation on~$\ov\fM_d^{\R}(\bfH)_I$.
 
\begin{lmm}\label{orientcomp_lmm3}
Let $d,k,n\!\in\!\Z^+$, $\bfH\!=\!(H_1,\ldots,H_k)$ be a general tuple of complex 
linear subspaces of~$\P^{2n-1}$ of complex codimensions $c_1,\ldots,c_k$,
respectively, and $I\!\subset\!\{1,\ldots,k\}$.
The orientations of $\ov\fM_d^{\R}(\bfH)\!\equiv\!\ov\fM_d^{\R}(\bfH)_{\eset}$
and $\ov\fM_d^{\R}(\bfH)_I$ are the same if and only~if the~set
$\{i\!\in\!I\!:\,c_i\!\in\!2\Z\}$
is of even cardinality.
\end{lmm}

\begin{proof}
By the transversality assumption, 
\BE{normbndl_e3} \tnd\ev_I\!: 
\frac{T(\fM_k^{\R}(d)_I)|_{\fM_d^{\R}(\bfH)_I}}{T(\fM_d^{\R}(\bfH)_I)}
\lra \ev_I^*\frac{T((\P^{2n-1})^k)|_{\lr\bfH_I}}{T(\lr\bfH_I)}\EE
is an isomorphism of vector bundles.
The orientation on the right-hand side of~\eref{normbndl_e3} induced by the complex orientations
of~$(\P^{2n-1})^k$ and $\lr\bfH_I$ induce an orientation on 
the left-hand side of~\eref{normbndl_e3}.
Along with the $I$-orientation on $\fM_k^{\R}(d)$,
the latter induces an orientation on~$\fM_d^{\R}(\bfH)_I$.
By the Chain Rule, $\tnd\ev_I\!=\!\tnd\Th_I\!\circ\!\tnd\ev$.
The sign of the isomorphism
$$\tnd\Th_I\!:\frac{(T(\P^{2n-1})^k)|_{\lr\bfH}}{T(\lr\bfH)}
\lra\Th_I^*\frac{T((\P^{2n-1})^k)|_{\lr\bfH_I}}{T(\lr\bfH_I)}$$
is $(-1)$ to the cardinality of the set $\{i\!\in\!I\!:\,c_i\!\not\in\!2\Z\}$.
The $I$-orientation on $\ov\fM_k^{\R}(d)$
differs from the canonical one by~$(-1)^{|I|}$.
Combining the two signs, we obtain the claim.
\end{proof}

\noindent
If $d\!=\!2d_1\!+\!d_2$ and $\{1,\ldots,k\}\!=\!I^+\!\sqcup\!J\!\sqcup\!I^-$,
let
$$\cN_{d_1,d_2;I^+,J,I^-}(\bfH)\subset \cN_{d_1,d_2}\cap \ov\fM_d^{\R}(\bfH)$$
be the subset consisting of maps from marked three-component domains so that 
the central component carries the marked points in the pairs indexed by~$J$,
one of the other components carries the first points in the pairs indexed by~$I^+$,
and the third component carries the first points in the pairs indexed by~$I^-$.
With notation as at the beginning of this section, 
we will associate~$I^+$ with the space of bubble components~$u^{\C}$ 
used to orient $\cN_{d_1,d_2;I^+,J,I^-}(\bfH)$;
these bubble components now carry marked points indexed by $I^+\!\sqcup\!I^-$,
in addition to the marked point corresponding to the node.
As in complex GW-theory, a small modification of 
the gluing map~\eref{Phidfn_e3} gives rise to a gluing map 
$$\Phi_{\bfH}\!: U_{\bfH}\lra   \ov\fM_d^{\R}(\bfH),$$
where $U_{\bfH}\!\subset\!L$ is a neighborhood of the zero section
in $L\!\lra\!\cN_{d_1,d_2;I^+,J,I^-}(\bfH)$.
If $k\!\ge\!4$, let
$$f_{1234}\!:\ov\fM_k^{\R}(d)\lra\ov\cM_{0,4}$$
be the projection onto the first marked points in the first four conjugate pairs.

\begin{crl}\label{sign_crl}
Let $d,k,n\!\in\!\Z^+$ be such that $k\!\ge\!4$ and 
$\bfH\!=\!(H_1,\ldots,H_k)$ be a general tuple of complex 
linear subspaces of~$\P^{2n-1}$ of complex codimensions $c_1,\ldots,c_k$,
respectively.
Suppose $d_1\!\in\!\Z^{\ge0}$ and $d_2\!\in\!\Z^+$ are such that 
$d\!=\!2d_1\!+\!d_2$ and 
$I^+,I^-,J\!\subset\!\{1,\ldots,k\}$ form a partition
of $\{1,\ldots,k\}$ such~that 
\BE{IJcond_e}
\big|I^+\cap\{1,2,3,4\}\big|=2 \qquad\hbox{or}\qquad
\big|I^-\cap\{1,2,3,4\}\big|=2.\EE
Let $\cN_{d_1,d_2;I^+,J,I^-}$ be oriented as in the paragraph after Lemma~\ref{gluing_lmm}.
\begin{enumerate}[label=(\arabic*),leftmargin=*]
\item If $J\!\cap\!\{1,2,3,4\}\!\neq\!\eset$, the sequence 
\BE{ses_e}0\lra T\big(\cN_{d_1,d_2;I^+,J,I^-}(\bfH) \big)
\lra T\big(\ov\fM_d^{\R}(\bfH)\big)\big|_{\cN_{d_1,d_2;I^+,J,I^-}(\bfH)}
\stackrel{\tnd f_{1234}}{\lra} f_{1234}^*T\ov\cM_{0,4}\lra0\EE
of vector bundles over $\cN_{d_1,d_2;I^+,J,I^-}(\bfH)$ is exact;
it is compatible with the canonical orientations if and only~if 
\BE{signcases_e}
\big|\{i\!\in\!I^-\!:\,c_i\!\in\!2\Z\big\}|+nd_1~
\begin{cases}
\in2\Z,&\hbox{if}~|I^+\!\cap\!\{1,2,3,4\}|\!=\!2;\\
\not\in2\Z,&\hbox{if}~|I^-\!\cap\!\{1,2,3,4\}|\!=\!2.
\end{cases}\EE
\item If $J\!\cap\!\{1,2,3,4\}\!=\!\eset$, the image of
 a fiber of $U\!\lra\!\cN_{d_1,d_2;I^+,J,I^-}(\bfH)$ 
under $f_{1234}\!\circ\!\Phi_{\bfH}$
is of real dimension~1.
\end{enumerate}
\end{crl}

\begin{rmk}
The requirement \eref{IJcond_e} insures that $f_{1234}$ is constant along 
$\cN_{d_1,d_2;I^+,J,I^-}(\bfH)$ and so the composition of the two arrows
in~\eref{ses_e} is trivial.
The conclusion of Corollary~\ref{sign_crl} is compatible with 
changing the distinguisged conjugate component in the paragraph after 
Lemma~\ref{gluing_lmm}
(which interchanges $I^+$ and~$I^-$ and thus the two cases on the right-hand side
of~\eref{signcases_e}) for the following reason.
Let 
$$\ov\fM_{d_1}^{\C}(\bfH;I^+,I^-)=
\big\{u\!\in\!\ov\fM_{\{0\}\sqcup I^+\sqcup I^-}^{\C}(d_1)\!:~
\ev_i(u)\!\in\!H_i~\forall~i\!\in\!I^+,~\ev_i(u)\!\in\!\ov{H_i}~\forall~i\!\in\!I^-
\big\}.$$
The space $\cN_{d_1,d_2;I^+,J,I^-}(\bfH)$ is oriented as the preimage of 
the cycle
$$\ev_0\!:\ov\fM_{d_1}^{\C}(\bfH;I^+,I^-)\lra \P^{2n-1}$$
by the evaluation map at the marked point of $\fM_{\{0\}\sqcup J}^{\R}(d_2)$
corresponding to the chosen node.
Interchanging $I^+$ and~$I^-$ replaces this cycle and the evaluation map
with their conjugates, as  before Lemma~\ref{orientcomp_lmm3}.
If the cardinalities of the sets $\{i\!\in\!I^{\pm}\!:\,c_i\!\in\!2\Z\}$ are 
of the same parity, the complex dimension~of $\ov\fM_{d_1}^{\C}(\bfH;I^+,I^-)$ 
is odd and so the codimension of the cycle~$\ev_0$ above is even.
By the same argument as in the proof of Lemma~\ref{orientcomp_lmm3},
the orientation of $\cN_{d_1,d_2;I^+,J,I^-}(\bfH)$ thus changes,
as expected from the change in the validity of~\eref{signcases_e} in this case. 
If the cardinalities of the sets $\{i\!\in\!I^{\pm}\!:\,c_i\!\in\!2\Z\}$ are 
of different parities, the codimension of the cycle~$\ev_0$ above is odd.
Interchanging $I^+$ and~$I^-$  then does not change 
the orientation of $\cN_{d_1,d_2;I^+,J,I^-}(\bfH)$, 
as expected from the validity of~\eref{signcases_e} not changing in this case.
\end{rmk}

\begin{proof}[{\bf\emph{Proof of Corollary~\ref{sign_crl}}}]
(1) By Lemma~\ref{gluing_lmm}, the gluing map 
$$\Phi_{\bfH}\!: U_{\bfH}\lra   \ov\fM_d^{\R}(\bfH)_{I^-}$$
is orientation-preserving if and only if $nd_1$ is even.
The differential 
\BE{fgldiff_e3} \tnd(f_{1234}\!\circ\!\Phi_{\bfH})\!: L\lra 
\big\{f_{1234}\!\circ\!\Phi_{\bfH}\big\}^*T\cM_{0,4}\EE
is the composition of the differential for smoothing the nodes in 
$\ov\fM_k^{\C}(d)$,
$$\tnd(f_{1234}\!\circ\!\Phi^{\C})\!: L\!\oplus\!L'\lra 
\big\{f_{1234}\!\circ\!\Phi^{\C}\big\}^*T\cM_{0,4}\,,$$
where $L'$ is the analogue of $L$ for the second node, with the embedding
$$L\lra L\oplus L', \qquad \ups\lra \big(\ups,\tnd\eta_u(\ups)\big);$$
see the last part of Section~\ref{lmmpf_sec}.
The restriction of the latter differential to the component, $L$ or~$L'$,
corresponding to the node separating off two of the marked points~$\{1,2,3,4\}$
is a $\C$-linear isomorphism, while the restriction to the other component is trivial.
If $|I^+\!\cap\!\{1,2,3,4\}|\!=\!2$, the former component is~$L$ and
\eref{fgldiff_e3} is an orientation-preserving map.
If $|I^-\!\cap\!\{1,2,3,4\}|\!=\!2$,  the former component is~$L'$ and
\eref{fgldiff_e3} is an orientation-reversing map.
Combining these two observations, we find~that 
the sequence 
$$0\lra T\big(\cN_{d_1,d_2;I^+,J,I^-}(\bfH) \big)
\lra T\big(\ov\fM_d^{\R}(\bfH)_{I^-}\big)\big|_{\cN_{d_1,d_2;I^+,J,I^-}(\bfH)}
\stackrel{\tnd f_{1234}}{\lra} f_{1234}^*T\ov\cM_{0,4}\lra0$$
of vector bundles over $\cN_{d_1,d_2;I^+,J,I^-}(\bfH)$ is exact;
it is compatible with the orientations if and only~if 
$$nd_1~\begin{cases}
\in2\Z,&\hbox{if}~|I^+\!\cap\!\{1,2,3,4\}|\!=\!2;\\
\not\in2\Z,&\hbox{if}~|I^-\!\cap\!\{1,2,3,4\}|\!=\!2.
\end{cases}$$
Combining this with Lemma~\ref{orientcomp_lmm3}, we obtain the first claim of 
Corollary~\ref{sign_crl}.\\

\noindent
(2) If $J\!\cap\!\{1,2,3,4\}\!=\!\eset$, the morphism 
$$f_{1234}\!:\ov\fM_{2k}^{\C}(d)\lra \ov\cM_{0,4}$$
is locally of the form 
$$L\oplus L'\lra \ov\cM_{0,4}, \qquad (\ups,\ups')\lra a\ups\ups',$$
for some $a$ dependent only on~$\cN_{d_1,d_2}$.
Thus, the restriction of~$f$ to $\ov\fM_k^{\R}(d)$ is locally of 
the~form
$$L\lra \ov\cM_{0,4}, \qquad \ups\lra a\ups\bar\ups,$$
which implies the last claim of Corollary~\ref{sign_crl}.
\end{proof}

\section{Comparison of orientations}
\label{lmmpf_sec}

\noindent
We now verify Lemma~\ref{gluing_lmm} by explicitly describing and comparing
the relevant orientations.
This argument is fundamentally different from the proof of 
\cite[Lemma~5.1]{GZ4}.\\

\noindent
Let $\Si$ be the nodal surface consisting of three components: 
\begin{enumerate}[label=(\arabic*),leftmargin=*]
\item $\Si_0\!=\!\P^1$ with nodes at $[c,1]$ and $[1,-c']$ for some $c,c'\!\in\!\C^*$
with $cc'\!\neq\!-1$;
\item $\Si^+\!=\!\P^1$ with the node at $[1,0]$, which is joined to $\Si_0$ at~$[c,1]$, and
\item $\Si^-\!=\!\P^1$ with the node at $[0,1]$, which is joined to $\Si_0$ at~$[1,-c']$. 
\end{enumerate}
A holomorphic map $u\!:\Si\!\lra\!\P^{m-1}$ corresponds to three maps:
\begin{enumerate}[label=(\arabic*),leftmargin=*]
\item $u_0\!:\P^1\!\lra\!\P^{m-1}$ typically given~by
$$[x,y]\lra  \bigg[A_1\!\prod_{r=1}^{d_0}(x\!-\!a_{1;r}y),\ldots,
A_m\!\!\prod_{r=1}^{d_0}(x\!-\!a_{m;r}y)\bigg],$$
\item $u^+\!:\P^1\!\lra\!\P^{m-1}$ typically given~by
$$[x,y]\lra \bigg[B_1\!\prod_{r=1}^{d^+}(x\!-\!b_{1;r}y),\ldots,
B_m\!\!\prod_{r=1}^{d^+}(x\!-\!b_{m;r}y)\bigg],$$
\item $u^-\!:\P^1\!\lra\!\P^{m-1}$ typically given~by
$$[x,y]\lra  \bigg[B_1'\!\prod_{r=1}^{d^-}(b_{1;r}'x\!+\!y),\ldots,
B_m'\!\!\prod_{r=1}^{d^-}(b_{m;r}'x\!+\!y)\bigg],$$
\end{enumerate}
for some $A_i,B_i,B_i',a_{i;r},b_{i;r},b_{i;r}'\!\in\!\C^*$ such that
\begin{gather*}
\bigcap_{r=1}^m\{a_{i;r}\!:\,r\!=\!1,\ldots,d_0\},
\bigcap_{r=1}^m\{b_{i;r}\!:\,r\!=\!1,\ldots,d^+\},
\bigcap_{r=1}^m\{b_{i;r}'\!:\,r\!=\!1,\ldots,d^-\}=\eset,\\
\begin{split}
[B_1,\ldots,B_m]&=\bigg[A_1\!\prod_{r=1}^{d_0}(c\!-\!a_{1;r}),\ldots,
A_m\!\!\prod_{r=1}^{d_0}(c\!-\!a_{m;r})\bigg],\qquad\hbox{and}\\
[B_1',\ldots,B_m']&=
\bigg[A_1\!\prod_{r=1}^{d_0}(1\!+\!a_{1;r}c'),\ldots,A_m\!\!\prod_{r=1}^{d_0}(1\!+\!a_{m;r}c')
\bigg].
\end{split}
\end{gather*}
The intersection conditions above are equivalent to the condition that  
the polynomials in each of the three sets describing~$u_0,u^+,u^-$ have no common factor;
the other two conditions are equivalent to 
$u^+([1,0])\!=\!u_0([c,1])$ and $u^-([0,1])\!=\!u_0([1,-c'])$.\\

\noindent
Deformations of maps of the form~$u$ above are described~by the holomorphic maps
\begin{equation*}\begin{split}
[x,y]&\lra  \bigg[A_1\!\prod_{r=1}^{d_0}(x\!-\!a_{1;r}y)
\prod_{r=1}^{d^+}\!\big(x\!-\!(c\!+\!b_{1;r}\ups)y\big)
\prod_{r=1}^{d^-}\!\big((c'\!+\!b_{1;r}'\ups')x\!+\!y\big),\ldots,\\
&\hspace{1.5in}
A_m\!\!\prod_{r=1}^{d_0}(x\!-\!a_{m;r}y)
\prod_{r=1}^{d^+}\!\big(x\!-\!(c\!+\!b_{m;r}\ups)y\big)
\prod_{r=1}^{d^-}\!\big((c'\!+\!b_{m;r}'\ups')x\!+\!y\big)\bigg],
\end{split}\end{equation*}
with $\ups,\ups'\!\in\!\C^*$ corresponding to the smoothings of the two nodes.\\

\noindent
We next take $m\!=\!2n$.
For $d\!\in\!\Z^+$, let
\begin{equation*}\begin{split}
\De_{n,d}^{\eta}&=\bigg\{
\big([a_{1;1},\ldots,a_{1;d}],\ldots,[a_{n;1},\ldots,a_{n;d}])\!\in\!
(\tn{Sym}^d\C)^n\!:\\
&\hspace{1in}
\bigcap_{r=1}^n\{a_{i;r}\!:\,r\!=\!1,\ldots,d\}\cap
\bigcap_{r=1}^n\{-1/\ov{a_{i;r}}\!:\,r\!=\!1,\ldots,d\}=\eset \bigg\}.
\end{split}\end{equation*}
A typical $(\eta_{2n},\eta)$-real degree $d_0\!=\!d_2$ holomorphic map
$u^{\R}\!\equiv\!u_0$ from~$\P^1$ to~$\P^{2n-1}$ is of the~form
$$[x,y]\lra \bigg[A_1\!\!\prod_{r=1}^{d_2}\!(x\!-\!a_{1;r}y),
\bar{A}_1\!\!\prod_{r=1}^{d_2}\!(\ov{a_{1;r}}x\!+\!y),\ldots, 
A_n\!\!\prod_{r=1}^{d_2}\!(x\!-\!a_{n;r}y),
\bar{A}_n\!\!\prod_{r=1}^{d_2}\!(\ov{a_{n;r}}x\!+\!y)\bigg]$$
for some $A_i,a_{i;r}\!\in\!\C^*$ with
$$\big([a_{1;1},\ldots,a_{1;d_2}],\ldots,[a_{n;1},\ldots,a_{n;d_2}])\!\in\!
(\tn{Sym}^{d_2}\C)^n-\De_{n,d_2}^{\eta}\,.$$
If $u^{\C}\!\equiv\!u^+$ is described as in the first paragraph of this section with
$d^+\!=\!d_1$, $u^-\!=\!\eta_{2n}\!\circ\!u^{\C}\!\circ\!\eta$
is given~by
$$[x,y]\lra
\bigg[-\ov{B_2}\!\prod_{r=1}^{d_1}(\ov{b_{2;r}}x\!+\!y),
\ov{B_1}\!\prod_{r=1}^{d_1}(\ov{b_{1;r}}x\!+\!y), \ldots,
-\ov{B_{2n}}\!\prod_{r=1}^{d_1}(\ov{b_{2n;r}}x\!+\!y),
\ov{B_{2n-1}}\!\prod_{r=1}^{d_1}(\ov{b_{2n-1;r}}x\!+\!y)\bigg].$$
The resulting map $u\!:\Si\!\lra\!\P^{2n-1}$ is $(\eta_{2n},\eta)$-real if $c'\!=\!\bar{c}$. 
In such a case, the restriction of the gluing map~\eref{Phidfn_e3} to an open subspace of~$U$
can be taken to~be
\begin{equation*}\begin{split}
\ups&\stackrel{\Phi}{\lra} \bigg[A_1\!\!\prod_{r=1}^{d_2}\!(x\!-\!a_{1;r}y)
\prod_{r=1}^{d_1}\!\big( (x\!-\!(c\!+\!b_{1;r}\ups)y)
\big(\ov{(c\!+\!b_{2;r})\ups}\,x\!+\!y\big)\big),\\
&\hspace{.45in}
\bar{A}_1\!\!\prod_{r=1}^{d_2}\!(\ov{a_{1;r}}x\!+\!y)
\prod_{r=1}^{d_1}\!\big( (x\!-\!(c\!+\!b_{2;r}\ups)y)
\big(\ov{(c\!+\!b_{1;r}\ups)}\,x\!+\!y\big)\big),\ldots\bigg],
\end{split}\end{equation*}
with $\ups\!\in\!\C^*$ corresponding to an element of~$L$
(based on the complex case in the previous paragraph).\\

\noindent
As explained in \cite[Section~2.1]{Teh}, an orientation on 
$\fM_0^{\R}(d)$ is equivalent 
to an orientation on the space $\wt\fM_0^{\R}(d)$ of parametrized real maps.
The latter is determined by the~map
\begin{gather*}
\big((\tn{Sym}^d\C)^n-\De_{n,d}^{\eta}\big)\!\times \R\P^{2n-1}\lra\wt\fM_0^{\R}(d)\,,\\
\begin{split}
&\big([a_{1;1},\ldots,a_{1;d}],\ldots,[a_{n;1},\ldots,a_{n;d}],[A_1,\ldots,A_n]\big)\lra\\
&\hspace{.5in}\bigg[A_1\!\prod_{r=1}^d\!(x\!-\!a_{1;r}y),
\bar{A}_1\!\prod_{r=1}^d\!(\ov{a_{1;r}}x\!+\!y),\ldots, 
A_n\!\prod_{r=1}^d\!(x\!-\!a_{n;r}y),
\bar{A}_n\!\prod_{r=1}^d\!(\ov{a_{n;r}}x\!+\!y)\bigg],
\end{split}\end{gather*}
where $\R\P^{2n-1}\!\equiv\!(\C^n\!-\!\{0\})/\R^*$ and $[x,y]\!\in\!\P^1$. 
This map is an isomorphism over the open subset
of~$\wt\fM_0^{\R}(d)$ consisting of maps~$u$ such that 
$u([1,0])$ does not lie in any of the coordinate subspaces of~$\P^{2n-1}$;
see \cite[Section~5.2]{Teh}.\\

\noindent
For $c\!\in\!\C^*$ as above, $i\!=\!1,2$, and 
$b\!\in\!\C^*$ with $|b|\!<\!|c|$, let
$$h_{c;i}(b)=\begin{cases}c\!+\!b,&\hbox{if}~i\!\not\in\!2\Z;\\
\ov{(c\!+\!b)}^{\,-1},&\hbox{if}~i\!\in\!2\Z.
\end{cases}$$
The explicit gluing map~$\Phi$ described above locally corresponds to the~map
$$\wt\Phi\!=\!(\ti\Phi_1,\ti\Phi_2)\!: 
\big(\C^{d_1})^{2n}\times (\C^{d_2})^n\times\R\P^{2n-1} \lra  (\C^d)^n\times\R\P^{2n-1}\,, $$
where $d\!=\!2d_1\!+\!d_2$, given~by
\begin{equation*}\begin{split}
\wt\Phi_{1;i;r}
\big((b_{j;s})_{\begin{subarray}{l}j\le2n\\ s\le d_1\end{subarray}},
(a_{j;s})_{\begin{subarray}{l}j\le n\\ s\le d_2\end{subarray}},
[A_1,\ldots,A_n]\big)&=
\begin{cases}
h_{c;1}(b_{2i-1;(r+1)/2}),&\hbox{if}~r\!\le\!2d_1,~r\!\not\in\!2\Z;\\
h_{c;2}(b_{2i;r/2}),&\hbox{if}~r\!\le\!2d_1,~r\!\in\!2\Z;\\
a_{i;r-2d_1},&\hbox{if}~r\!>\!2d_1;
\end{cases}\\
\wt\Phi_{2;i}
\big((b_{j;s})_{\begin{subarray}{l}j\le2n\\ s\le d_1\end{subarray}},
(a_{j;s})_{\begin{subarray}{l}j\le n\\ s\le d_2\end{subarray}},
[A_1,\ldots,A_n]\big)&=A_i\bigg/\prod_{r=1}^{d_1}h_{c;2}(b_{2i;r}).
\end{split}\end{equation*}
The sign of $\wt\Phi$ is $(-1)^{nd_1}$, which establishes Lemma~\ref{gluing_lmm}.\\

\vspace{.4in}

\vbox{
\noindent
{\it Department of Mathematics, Princeton University, Princeton, NJ 08544}\\
{\it Current address: Institut de Math\'ematiques de Jussieu - Paris Rive Gauche,
Universit\'e Pierre et Marie Curie,  4~Place Jussieu,
75252 Paris Cedex 5, France}\\

\noindent
{\it Department of Mathematics, Stony Brook University, Stony Brook, NY 11794\\
azinger@math.stonybrook.edu}}


\begin{thebibliography}{99}

\bibitem{ABL} A.~Arroyo, E.~Brugall\'e, and L.~L\'opez de Medrano, 
{\it Recursive formulas for Welschinger invariants of the projective plane}, 
IMRN (2011), no.~5, 1107–-1134

\bibitem{Cho} C.-H.~Cho, 
\emph{Counting real J-holomorphic discs and spheres in dimension four and six}, 
J.~Korean Math.~Soc.~45 (2008), no.~5, 1427–-1442

\bibitem{Teh} M.~Farajzadeh Tehrani,
\emph{Counting genus zero real curves in symplectic manifolds},
math/1205.1809, to appear in Geom.~Top.


\bibitem{GM} A.~Gathmann and H.~Markwig, 
{\it The Caporaso-Harris formula and plane relative Gromov-Witten invariants 
in tropical geometry}, Math.~Ann.~338 (2007), no.~4, 845–-868

\bibitem{Ge2} P.~Georgieva,
\emph{Open Gromov-Witten invariants in the presence of an anti-symplectic involution},
math/1306.5019 

\bibitem{GZ1} P.~Georgieva and A.~Zinger,
\emph{The moduli space of maps with crosscaps:
Fredholm theory and orientability},
Comm.~Anal.~Geom.~23 (2015), no.~3, 81--140

\bibitem{GZ4} P.~Georgieva and A.~Zinger,
\emph{Enumeration of real curves in $\C\P^{2n-1}$ and
a WDVV relation for real Gromov-Witten invariants}, math/1309.4079

\bibitem{IKS13a} I.~Itenberg, V.~Kharlamov, and E.~Shustin, 
{\it Welschinger invariants of small non-toric Del Pezzo surfaces}, 
J.~EMS 15 (2013), no.~2, 539–-594

\bibitem{IKS13b} I.~Itenberg, V.~Kharlamov, and E.~Shustin, 
{\it Welschinger invariants of real del Pezzo surfaces of degree $\ge3$}, 
Math.~Ann.~355 (2013), no.~3, 849–-878


\bibitem{MirSym} K.~Hori, S.~Katz, A.~Klemm, R.~Pandharipande,
R.~Thomas, C.~Vafa, R.~Vakil, and E.~Zaslow, 
\emph{Mirror Symmetry}, Clay Math.~Inst., AMS~2003

\bibitem{KM} M.~Kontsevich and Y.~Manin, 
\emph{Gromov-Witten classes, quantum cohomology, and enumerative geometry}, 
Comm.~Math.~Phys.~164 (1994), no.~3, 525–-562


\bibitem{PSW} R.~Pandharipande, J.~Solomon, and J.~Walcher,
\emph{Disk enumeration on the quintic 3-fold}, 
J.~Amer.~Math.~Soc. 21 (2008), no.~4, 1169--1209

\bibitem{RT} Y.~Ruan and G.~Tian, 
\emph{A mathematical theory of quantum cohomology}, 
J.~Differential Geom.~42 (1995), no.~2, 259–-367

\bibitem{Sol} J.~Solomon,  
\emph{Intersection theory on the moduli space of holomorphic curves with 
Lagrangian boundary conditions}, math/0606429

\bibitem{Sol2} J.~Solomon, 
{\it A differential equation for the open Gromov-Witten potential},
pre-print 2007

\bibitem{Wel1} J.-Y.~Welschinger, 
\emph{Invariants of real symplectic 4-manifolds and lower bounds in real enumerative geometry}, 
Invent.~Math.~162 (2005), no.~1, 195–-234

\bibitem{Wel2} J.-Y.~Welschinger,  
\emph{Spinor states of real rational curves in real algebraic convex 3-manifolds 
and enumerative invariants},
Duke Math.~J.~127 (2005), no.~1, 89–-121


\end{thebibliography}
\end{document}